\documentclass{conm-p-l}

\usepackage{amssymb}

\usepackage{graphicx}

\usepackage[cmtip,all]{xy}

\usepackage{}

\newtheorem{theorem}{Theorem}[section]
\newtheorem{lemma}[theorem]{Lemma}
\newtheorem{prop}[theorem]{Proposition}
\newtheorem{corollary}[theorem]{Corollary}

\theoremstyle{definition}
\newtheorem{defn}[theorem]{Definition}

\theoremstyle{remark}
\newtheorem{remark}[theorem]{Remark}

\numberwithin{equation}{section}

\DeclareMathOperator{\Pic}{Pic}
\DeclareMathOperator{\Mod}{-Mod}
\newcommand{\Hom}{\mathrm{Hom}}
\newcommand{\Cok}{\mathrm{Cok}}
   \newcommand{\Ker}{\mathrm{Ker}}
   \newcommand{\Image}{\mathrm{Im}}
\newcommand{\Ext}{\mathrm{Ext}}

\newcommand{\red}{\mathrm{red}}
\newcommand{\St}{\mathrm{St}}
 
\newcommand{\lra}{\longrightarrow}
   \newcommand{\llra}[1]{\stackrel{#1}{\longrightarrow}}

\newcommand{\iso}{\;\cong\;}
\newcommand{\congruent}{\;\equiv\;}    
\newcommand{\hequiv}{\;\simeq\;}       
\newcommand{\modmod}{/\!\!/ }

\newcommand{\Mdef}[2]{\newcommand{#1}{\relax \ifmmode #2 \else $#2$\fi}}
\Mdef{\cA}{\mathcal{A}}
\Mdef{\Z}{\mathbf{Z}}
\Mdef{\F}{\mathbf{F}}
\Mdef{\LL}{\mathbf{L}}

   \newcommand{\qqed}{\qed \\[1ex]}
\newcommand{\tensor}{\otimes}
   \newcommand{\smsh}{\wedge}        
\newcommand{\dirsum}{\oplus}
   \newcommand{\bigdirsum}{\bigoplus}
\newcommand{\st}{\; | \; }
\Mdef{\cC}{\mathcal{C}}
   \Mdef{\RP}{{\mathbf{R}}{\mathrm P}}

\begin{document}

\title{Idempotents, Localizations and  Picard groups of A(1)-modules}

\author{Robert R. Bruner}
\address{Department of Mathematics\\
	  Wayne State University\\
	  Detroit, Michigan  48202\\
	  USA
	 }
\email{rrb@math.wayne.edu}


\begin{abstract}
We analyze the stable isomorphism type of polynomial rings on degree $1$
generators as modules over
the subalgebra $\cA(1) = \langle Sq^1, Sq^2\rangle$ of the mod $2$ Steenrod algebra.
Since their augmentation ideals are 
$Q_1$-local, we do this by studying the $Q_i$-local 
subcategories and the associated Margolis localizations.  
The periodicity exhibited by such modules 
reduces the calculation to one that is finite.
We show that these
are the only localizations which preserve tensor products, by first
computing
the Picard groups of these subcategories and using them to determine
all idempotents in the stable category of bounded-below $\cA(1)$-modules.
We show that the Picard groups
of the whole category are detected in the local Picard groups, and show
that every bounded-below $\cA(1)$-module is uniquely expressible
as an extension of a $Q_0$-local module by a $Q_1$-local module, up
to stable equivalence.  
\end{abstract}

\maketitle

\tableofcontents

\section{Introduction}

Let $H^*$ denote reduced mod $2$ cohomology.
We organize into a systematic framework the ideas that have been used
to analyze the $\cA(1)$-module structure of
$H^*BV_+ = \F_2[x_1,\ldots,x_n]$, 
where $V$ is an elementary abelian 2-group of
rank $n$.
As always, this splits into a direct sum of tensor powers of the
rank 1 case, $H^*BC_2$.  Remarkably,
as an $\cA(1)$-module, the 
tensor powers of $H^*BC_2$ are stably equivalent to their algebraic loops
(syzygies).
This is a general phenomenon:  if $I$ is a stably idempotent module
over a finite dimensional Hopf algebra, i.e., if
$I \tensor I \hequiv I$, then $\Omega^n I \hequiv (\Omega I)^{\tensor(n)}$:
\begin{align*}
(\Omega I)^{\tensor(n)} & = \Omega I \tensor \cdots \tensor \Omega I \\
& \hequiv  \Omega^n(I^{\tensor(n)}) \\
& \hequiv  \Omega^n I.
\end{align*}
Localizations provide a ready source of idempotents:  since $\F_2$ is 
tensor idempotent, its Margolis localizations $\LL_i \F_2$ are as well.
It happens that $\Sigma H^*BC_2 = \Omega \LL_1 \F_2$.

Our main results are as follows.

We call a bounded-below
module $Q_k$-local if its only non-zero Margolis homology is
with respect to $Q_k$ (Definition~\ref{Qklocaldef}).  If $M$ is $Q_0$-local then
$\Omega M \hequiv \Sigma M$, while if $M$ 
is $Q_1$-local then
$\Omega^4 M \hequiv \Sigma^{12} M$ (Theorems~\ref{onefold} and 
\ref{fourthloops}).  

We define modules $R$ and $P_0$ closely related to $ H^*BC_2$
and  observe that 
$R$ is $Q_0$-local and $P_0$ is $Q_1$-local.  We  show
there is a unique non-split triangle
\[
\Sigma R \llra{\epsilon} \F_2 \llra{\eta} P_0
\]
(Proposition~\ref{key-sequence}).  It follows that these are 
Margolis localizations: 
$\LL_0 \F_2 \hequiv \Sigma R$ and
$\LL_1 \F_2 \hequiv P_0$.  They are therefore idempotent, and,
as observed above, their tensor powers coincide with their
algebraic loops, which therefore exhibit one and four-fold periodicity,
respectively.
Since $\Omega P_0 \hequiv \Sigma H^*BC_2$,
the tensor powers of $H^*BC_2$ exhibit four-fold periodicity.
This reduces the analysis of all their tensor powers to four
cases, which we carry out explicitly in Section~\ref{PntoPone}.

We then deduce the basic properties of the localizations, including
the fact that the natural triangle
\[
\xymatrix{
\LL_0 M
\ar_{\epsilon_M}[r]
&
M
\ar_{\eta_M}[r]
&
\LL_1 M
\\
}
\]
(Definition~\ref{Lidefn}) is the unique triangle of the form
\[
\xymatrix{
M_0
\ar_{\epsilon}[r]
&
M
\ar_{\eta}[r]
&
M_1 
\\
}
\]
in which each $M_i$ is $Q_i$-local (Theorem~\ref{uniquetriangle}).

We next show that the localizations $\LL_i\F_2$ and their suspensions and loops
account for the whole Picard group of the $Q_i$-local subcategories (with
no local finiteness hypotheses needed).  We show that if $\Pic^{(i)}$ denotes the 
Picard group of the category of bounded-below $Q_i$-local modules, then
\[
\begin{array}{lcl}
\Pic^{(0)}(E(1)) = \Z 
&
\mathrm{\phantom{a bit of space}}
&
\Pic^{(0)}(\cA(1)) = \Z 
\\
\Pic^{(1)}(E(1)) = \Z 
&
\mathrm{\phantom{a bit of space}}
&
\Pic^{(1)}(\cA(1)) = \Z \dirsum \Z/(4) 
\\
\end{array}
\]
with the $\Z/(4)$ due to the four-fold periodicity of the loops of $P_0$
(Theorems~\ref{Yuslemma} and \ref{Yustheorem} and Propositions~\ref{PiczeroE}
and \ref{PiczeroA}).
Next we show that the global Picard group is detected in the local ones:
the localization map
\[
\Pic \lra \Pic^{(0)} \dirsum \Pic^{(1)}
\]
is a monomorphism (Section~\ref{sec:Pichom}).  

We then show that the only bounded-below stably idempotent
$\cA(1)$-modules are those we have already seen 
(Theorem~\ref{idempotents}) so that we have found all
localizations of the form $L(M) = I \tensor M$,
$I$  stably idempotent.

The last section in the main body of the paper observes that there is an
idempotent, the Laurent series ring $L$, that is neither bounded-below
nor bounded-above.  
It shows that the Margolis localizations are more fundamental than
the Margolis homology:  $L$ is $Q_1$-local in the generalized sense
that $L \hequiv \LL_1 L$ (and $\LL_0 L \hequiv 0$) despite having
trivial $Q_1$ and $Q_0$ homology.

Finally, 
in an appendix, we give precise 
form to the stable equivalences we have been studying, in
the expectation that these will be useful
in studying the `hit problem':  the study of
the $\cA$ and $\cA(n)$ indecomposables
in $H^*BV$.  (See \cite{Ault}, \cite{AultSinger} or \cite{WalkerWood},
for recent work on this problem.)

Since many of these results are modern versions of older results,
a brief summary of their development seems in order.
The algebraic loops (syzygies) of $H^*BC_2$
were explicitly identified in Margolis (\cite[Chap. 23]{Margolis}),
but had already been visible as early as the 1968 paper \cite{GMM} by
Gitler, Mahowald and Milgram, though the periodicity
was not stated there.
The relation to the tensor powers of $H^*BC_2$ 
was the discovery of Ossa (\cite{Ossa}).   He
showed that $P = H^*BC_2$ is stably idempotent as a
module over the subalgebra $E(1) = E[Q_0,Q_1]$ of the Steenrod algebra, and
used this to show that if $V$ is an elementary abelian group then,
modulo  Bott torsion, the connective complex K-theory of $BV_+$ is the
completion of the Rees ring of the representation ring $R(V)$ with
respect to its augmentation ideal.   (This is not how he said it, and
his main focus was on related topological  results, but this is one way
of phrasing the first theorem in~\cite{Ossa}.)  He tried to extend this
to real connective K-theory, but there were flaws in his argument.
By 1992, Stephan Stolz (private communication) knew that the correct statement for the
real case was that $P^{\tensor(n+1)}$ was the $n^{\mathrm th}$ syzygy of $P$
in the category of $\cA(1)$-modules.  
In his unpublished 1995 Notre Dame PhD thesis, Stolz's student
Cherng-Yih Yu (\cite{Yu}))
gave a proof of this  together with the remarkable fact that these
$\cA(1)$-modules form the Picard group of the category of 
bounded-below,
$Q_1$-local $\cA(1)$-modules.
As with Ossa's result in the complex case, this should lead to a representation
theoretic description of
the real connective K-theory of $BV_+$ modulo Bott torsion.  However, this was found
by other means
in the author's joint work with John Greenlees (\cite[p. 177]{kobg}).
More recently, Geoffrey Powell has given descriptions of 
the real and complex connective K-homology and cohomology of $BV_+$
in~\cite{Powellcx} and \cite{Powellreal}.
His functorial approach provides significant simplifications.  
Some of the results here are used
in his work on the real case.
Most recently, Shaun Ault has made use of the results here in his
study~\cite{Ault} of the hit problem.

The present account is essentially  self contained.  In particular, we give
dramatically simplified calculations of the Picard groups of the 
local subcategories.   
The work has evolved fitfully over the years since \cite{RRB-Ossa},
to which it provides context and additional detail,
receiving one impetus from my joint work with John Greenlees
(\cite{kubg} and \cite{kobg}), 
another from questions asked by Vic Snaith (which led to~\cite{VSetal}),
and a more recent one from discussions
with Geoffrey Powell in connection with  \cite{Powellreal}.
I am grateful to Geoffrey Powell for many useful discussions while
working out some of these results
and to the University of Paris 13 for the opportunity to work on this
in May of 2012.

\section{Recollections}

We begin with some basic definitions and results about modules over 
finite sub-Hopf algebras
of the mod $2$ Steenrod algebra, in order to state clearly the hypotheses under which 
they hold.
The reader who is familiar with $\cA(1)$-modules should probably 
skip to the next section.

Let $\cA(n)$ be the subalgebra of the mod 2 Steenrod algebra $\cA$
generated by $\{Sq^{2^i} \st 0 \leq i \leq n\}$.  Thus
$\cA(0)$ is exterior on one generator, $Sq^1$, and
$\cA(1)$, generated by $Sq^1$ and $Sq^2$, is $8$ dimensional.

Let $E(n)$ be the exterior subalgebra of $\cA$ generated by
the Milnor primitives $\{ Q_i \st 0 \leq i \leq n\}$.
(Recall that
$Q_0 = Sq^1$ and $Q_{n} = Sq^{2^n} Q_{n-1} + Q_{n-1} Sq^{2^n}.$)
$E(n)$ is a sub-Hopf algebra of $\cA(n)$.


For $B=E(1)$, $\cA(1)$, or any finite sub-Hopf algebra of $\cA$,
let $B\Mod$ be the category of all graded
$B$-modules.
The category $B\Mod$ is abelian, complete, cocomplete, has enough projectives
and injectives, and has a symmetric monoidal product $\tensor = \tensor_{\F_2}$.  
Since $B$ is a Frobenius algebra, free, projective and injective are equivalent
conditions  in $B\Mod$.
(See Margolis (\cite{Margolis}), Chapters 12, 13 and 15, and in particular
Lemma 15.27 for details.)

The best results hold in the abelian subcategory 
$B\Mod^{b}$ of bounded-below $B$-modules.
It has enough projectives and injectives (\cite[Lemma 15.27]{Margolis}).
A module in $B\Mod^b$  is free, projective, or injective 
there iff it is so in $B\Mod$ (\cite[Lemma 15.17]{Margolis}).  


Since the algebras $B$ we are considering are Poincare duality algebras,
the following decomposition result  holds without restriction on $M$.
It will be useful in our discussion of stable isomorphism.

\begin{prop}
\label{reduced}
(\cite[Proposition 13.13 and p. 203]{Margolis})
A module $M$ in $B\Mod$ has an expression 
\[
M \iso F \dirsum M^\red,
\]
unique up to isomorphism, where $F$ is free and $M^\red$ has no free summands.
\end{prop}

\begin{defn}
\label{reduceddef}
We call $M^\red$ the {\em reduced} part of $M$.
\end{defn}

Note that we are not asserting that $M \mapsto M^\red$ is a functor,
or that there are {\em natural} maps $M \lra M^\red$ or $M^\red \lra M$.

\begin{defn}
\label{defred}
If $\cC$ is a subcategory of $B\Mod$ which contains the projective modules,
the {\em stable module category of} $\cC$,
written $\St(\cC)$, is the category 
with the same objects as $\cC$ and  with
morphisms replaced by 
their
equivalence classes modulo those which factor through a projective module.  
Let us write $M \simeq N$ to denote {\em stable isomorphism}, isomorphism
in $\St(B\Mod)$, and reserve
$M \iso N$ for isomorphism in $B\Mod$.
\end{defn}

Over a finite Hopf algebra like $B$, stable isomorphism simplifies.

\begin{prop}
\label{stableisosum}
(\cite[Proposition 14.1]{Margolis})
In $B\Mod$, modules $M$ and $N$ are stably isomorphic iff there exist free modules 
$P$ and $Q$ such that $M \dirsum P \iso  N \dirsum Q$.
\end{prop}

In $B\Mod^b$, stable isomorphism simplifies further.

\begin{prop}
\label{stableisored}
(\cite[Proposition 14.11]{Margolis}
Let $M$ and $N$ be modules in $B\Mod^b$.  
\begin{enumerate}
\item $M \simeq N$ iff $M^\red \iso N^\red$.
\item $f : M \lra N$ is a stable equivalence iff
\[
\xymatrix{
M^\red\,\,
\ar@{>->}[r]
&
M 
\ar^{f}[r]
&
N 
\ar@{->>}[r]
&
N^\red
}
\]
is an isomorphism in $B\Mod$.
\end{enumerate}
\end{prop}

Here, 
$\xymatrix{
M^\red \,\,
\ar@{>->}[r]
&
M}$
and
$\xymatrix{
N 
\ar@{->>}[r]
&
N^\red
}
$
are any maps which are part of
a splitting of $M$ and $N$, respectively, into a free summand and a reduced summand.

The preceding result holds for all finite Hopf algebras.
For modules over subalgebras $B$ of the mod $2$ Steenrod algebra,
the theorem of Adams and Margolis (\cite{Ad-Mar} or \cite[Theorem 19.6]{Margolis})
gives us a simple criterion for stable isomorphism in $B\Mod^b$.  
Recall that the Milnor primitives $Q_i$
satisfy $Q_i^2 = 0$, so that we may define $H(M,Q_i) = \Ker(Q_i)/\Image(Q_i)$.

\begin{theorem}
\label{stableiso}
Let $B=\cA(1)$ or $E(1)$.  Suppose that
 $f : M \lra N$ in $ B\Mod^b$.   
If $f$ induces isomorphisms $f_*:H(M,Q_i) \lra H(N,Q_i)$ for
$i =0$ and $i=1$, then $f$ is  a stable isomorphism.
\end{theorem}

In particular, if a bounded-below module $M$ has trivial $Q_0$ and $Q_1$ homology, 
then the map
$0 \lra M$ is a stable equivalence, and therefore $M$ is free.

\begin{remark}
\label{stableisofail}
The hypothesis that the modules
be bounded-below is needed for Theorem~\ref{stableiso} to hold:  
the Laurent series ring $\F_2[x,x^{-1}]$ is not free
over $E(1)$ or $\cA(1)$,
yet has trivial $Q_0$ and $Q_1$ homology.
\end{remark}

Margolis (\cite[Theorem 19.6.(b)]{Margolis}) gives a similar characterization
of stable isomorphism or modules over
any sub-Hopf algebra $B$ of the mod $2$ Steenrod algebra.

Finally, we consider the algebraic loops functor.  By Schanuel's Lemma, letting
$\Omega M$ be the kernel of an epimorphism from a projective module to $M$
gives a well defined module up to stable isomorphism.  To get functoriality,
the following definition is simplest.

\begin{defn}
\label{defomega}
Let $I = \Ker(B \lra \F_2)$ be the augmentation ideal of $B$.  Let
$\Omega M = I \tensor M$.
\end{defn}

Note that $ \Omega \F_2 \iso I $.
Similarly, we may define the inverse loops functor.

\begin{defn}
\label{Loopinverse}
Let $I^{-1} = \Cok(\F_2 \lra \Sigma^{-d} B))$ be the 
cokernel of 
the $d^{{th}}$ desuspension of the 
the inclusion of the socle
into $B$.  ($d$  is $4$ if $B=E(1)$, $6$ if $B=\cA(1)$.)   Let
$\Omega^{-1} M = I^{-1} \tensor M$.
\end{defn}

To see that the notation makes sense, recall
the `untwisting' isomorphism 
\[
\theta : B \tensor M \lra B \tensor \widehat{M},
\]
given by  $\theta(b\tensor m) = \sum b' \tensor b'' m$.
Here $B \tensor \widehat{M}$ is the free $B$-module on the underlying vector
space $\widehat{M}$ of $M$ and $\psi(b) = \sum b' \tensor b''$ is the
coproduct of $b$.
The inverse, $\theta^{-1}(b \tensor m) = \sum b' \tensor \chi(b'')m$,
where $\chi$ is the conjugation (antipode) of $B$.
This shows that tensoring with a free module gives a free module.

{\em In particular, tensor product is well defined in the stable module
category.}

Tensoring the short exact sequence $0 \lra I \lra B \lra \F_2 \lra 0$
with $I^{-1}$ shows that $I \tensor I^{-1}$ is
stably equivalent to $\F_2$.  

\begin{corollary}
We have stable
equivalences $\Omega \Omega^{-1} \simeq {\mathrm{Id}} \simeq \Omega^{-1}\Omega$. 
In general,  $\Omega^k\Omega^l \simeq \Omega^{k+l}$
for all integers $k$ and $l$.
\end{corollary}

Finally, we should note that the stable module category is triangulated.  For any
short exact sequence of modules
\[
0 \lra M_1 \lra M_2 \lra M_3 \lra 0
\]
there is an extension cocycle $\Omega M_3 \lra M_1$ (or equivalently
$M_3 \lra \Omega^{-1} M_1$) representing the extension class in
$\Ext_B^1(M_3,M_1)$.  The triangles in the stable module category
are the sequences
\[
\Omega M_3 \lra M_1 \lra M_2 \lra M_3
\]
and
\[
 M_1 \lra M_2 \lra M_3 \lra  \Omega^{-1} M_1.
\]
for the short exact sequences
\[
0 \lra M_1 \lra M_2 \lra M_3 \lra 0.
\]

\section{Periodicity}

We start by observing the periodicities which {\em local} $B$-modules obey,
for $B=E(1)$ or $\cA(1)$.
We shall restrict attention to the category $B\Mod^b$ of bounded-below
$B$-modules.

\begin{defn}  Let $B$ be either $E(1)$ or $\cA(1)$.
\label{Qklocaldef}
Call a $B$-module {\em $Q_k$-local} if $H(M,Q_i) = 0$ for $i \neq k$.
For $k\in\{0,1\}$, let $B\Mod^{(k)}$ be the full subcategory  of $B\Mod^b$
containing the $Q_k$-local modules.
\end{defn}

\begin{theorem}
\label{onefold}
If $M \in B\Mod^{(0)}$ then $\Omega M \hequiv \Sigma M$.
\end{theorem}

\begin{proof}
Evidently, $\cA(0)$ has a unique $B$-module (even, $\cA$-module) structure
compatible with its structure as a module over itself.
Tensor $M$ with the  short exact sequence of $B$-modules
\[
0 \lra \Sigma \F_2 \lra \cA(0) \lra \F_2 \lra 0.
\]
We obtain
\[
0 \lra \Sigma M \lra M \tensor \cA(0) \lra M \lra 0.
\]
By Theorem~\ref{stableiso} and the K\"unneth isomorphism for $Q_i$ homology, 
the module in the middle is free and the result follows.
\end{proof}

The $Q_1$-local case requires a bit of preparation.
Recall the notation $A\modmod B$ for the $A$-module $A\tensor_B \F_2$
when $B$ is a sub-(Hopf-)algebra of $A$.

\begin{defn}
\label{fourfold}
Define modules $F_i$ 
and maps $f_i : F_{i+1} \lra F_i$ 
for $i \in \Z$ by
$F_{i+4} = \Sigma^{12} F_i$,
$f_{i+4} = \Sigma^{12} f_i$,
$f_3 = Sq^2 Sq^3$ and the following:\\[1ex]
\scalebox{.82}{
$
\xymatrix@R=18pt{
0
&
\F_2
\ar@{=}[d]
\ar[l]
&
F_0
\ar[l]
\ar@{=}[d]
&
F_1
\ar_{f_0}[l]
\ar@{=}[d]
&
F_2
\ar_{f_1}[l]
\ar@{=}[d]
&
F_3
\ar_{f_2}[l]
\ar@{=}[d]
&
\Sigma^{12}\F_2
\ar@{=}[d]
\ar[l]
&
0
\ar[l]
\\
0
&
\F_2
\ar[l]
&
\cA(1)\modmod\cA(0)
\ar[l]
&
\Sigma^2 \cA(1)
\ar_(.4){Sq^2}[l]
&
\Sigma^4 \cA(1)
\ar_(.45){Sq^2}[l]
&
\Sigma^7 \cA(1)\modmod\cA(0)
\ar_(.55){Sq^3}[l]
&
\Sigma^{12}\F_2
\ar_(.4){Sq^2Sq^3}[l]
&
0
\ar[l]
\\
}
$
} 
\end{defn}

The following is an elementary calculation, originally due to Toda~\cite{Toda}.
The diagram in the proof of Proposition~\ref{descrMi} is sufficient to
prove it.

\begin{prop}
\label{fourfoldperiodicity}
The sequence in Definition~\ref{fourfold}
is exact.
\qqed
\end{prop}

Splicing this sequence and its suspensions, we obtain 
a complete (i.e., Tate) resolution of $\F_2$ by modules tensored up from $\cA(0)$:
the $F_{4i}$ and $F_{4i+3}$ are suspensions of $\cA(1)\tensor_{\cA(0)}\F_2$, while
the $F_{4i+1}$ and $F_{4i+2}$ are suspensions of $\cA(1)\tensor_{\cA(0)}\cA(0)$.

\[
\xymatrix@R=8pt{
\cdots
&
F_{-2}
\ar_{f_{-3}}[l]
&
F_{-1}
\ar_{f_{-2}}[l]
&
&
F_{0}
\ar_{f_{-1}}[ll]
\ar[ld]
&
F_1
\ar_{f_0}[l]
&
F_2
\ar_{f_1}[l]
&
F_3
\ar_{f_2}[l]
&
\cdots
\ar_{f_3}[l]
\\
&&&
\F_2
\ar[lu]
\\
}
\]

The cokernels in this sequence will play an important role.  They are the syzygies
of $\F_2$ with respect to the relative projective class of projectives relative
to the $\cA(0)$-split exact sequences.

\begin{defn}
\label{defMi}
Let $M_i = \Sigma^{-i} \Cok{f_i} $.
\end{defn}

We have inserted the suspension here to make later calculations run more
smoothly.
It is a simple matter to describe the $M_i$.

\begin{prop}
\label{descrMi}
For each $i\in \Z$, $M_{i+4} = \Sigma^{8} M_i$,
so the following suffice to determine all the $M_i$:
\begin{itemize}
\item $M_0 =  \F_2$,
\item $M_1 = \Sigma \cA(1)/(Sq^2)$,
\item $M_2 = \Sigma^2 \cA(1)/(Sq^3)$,
\item $M_3 = \Sigma^4 \cA(1)/(Sq^1,Sq^2Sq^3)$.
\end{itemize}
\end{prop}

\begin{proof}
The following diagram exhibits the $\Sigma^i M_i$ by open dots in the diagram
of $F_i$, or as solid dots in the diagram of $F_{i-1}$.\\[1ex]
\scalebox{.82}{
$
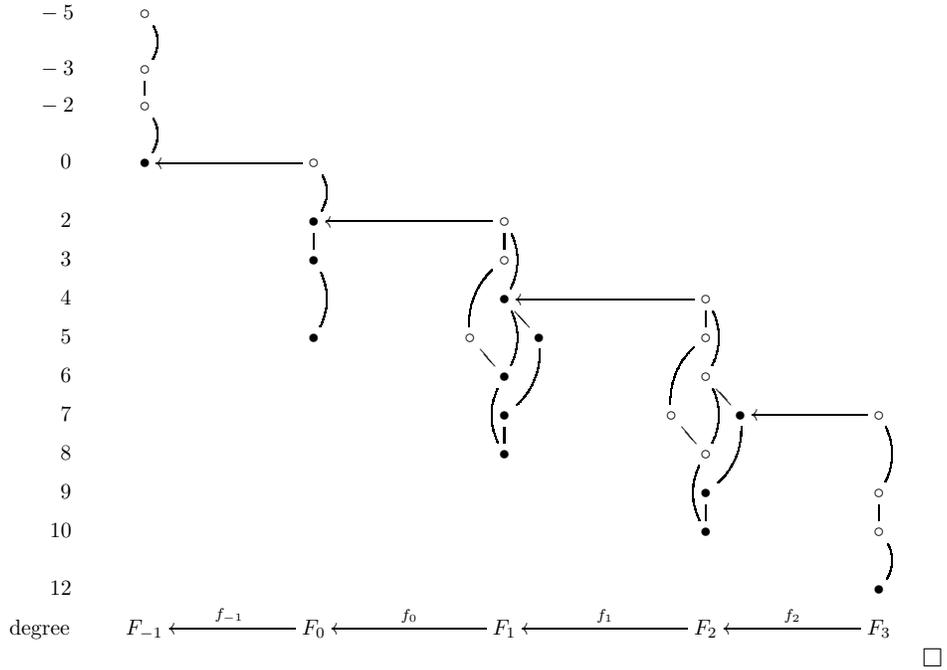
\xymatrix@R=3pt@C=2pt{
{\mathrm{\phantom{der}}} -5 & &&
\circ
\ar@{-}@/^.5pc/[dd]
\\
\\
{\mathrm{\phantom{der}}}
-3 &&&
\circ
\ar@{-}[d]
\\
{\mathrm{\phantom{der}}}
-2 &&&
\circ
\ar@{-}@/^.5pc/[dd]
\\
\\
{\mathrm{\phantom{degre}}}
0 &&&
\bullet
&&&&&&&&
\circ
\ar[llllllll]
\ar@{-}@/^.5pc/[dd]
\\
\\
{\mathrm{\phantom{deger}}}
2 &&&
&&&&&&&& \bullet
\ar@{-}[d]
&&&&&&&&& \circ
\ar[lllllllll]
\ar@{-}[d]
\ar@{-}@/^.5pc/[dd]
\\
{\mathrm{\phantom{deger}}}
3 &&&
&&&&&&&& \bullet
\ar@{-}@/^.5pc/[dd]
&&&&&&&&& \circ
\ar@{-}@/_.5pc/[ddl]
\\
{\mathrm{\phantom{deger}}}
4 &&&
&&&&&&&& 
&&&&&&&&& \bullet
\ar@{-}@/^.5pc/[dd]
\ar@{-}[dr]
&&&&&&&&&
\circ
\ar@{-}@/^.5pc/[dd]
\ar@{-}[d]
\ar[lllllllll]
\\
{\mathrm{\phantom{deger}}}
5 &&&
&&&&&&&& \bullet
&&&&&&&& \circ
\ar@{-}[dr]
&& \bullet
\ar@{-}@/^.5pc/[ddl]
&&&&&&&& \circ
\ar@{-}@/_.5pc/[ddl]
\\
{\mathrm{\phantom{deger}}}
6 &&&
&&&&&&&& &&&&&&&&& 
\bullet
\ar@{-}@/_.5pc/[dd]
&&&&&&&&&
\circ
\ar@{-}@/^.5pc/[dd]
\ar@{-}[dr]
\\
{\mathrm{\phantom{deger}}}
7 &&&
&&&&&&&& &&&&&&&&&
\bullet
\ar@{-}[d]
&&&&&&&&
\circ
\ar@{-}[dr]
&&
\bullet
\ar@{-}@/^.5pc/[ddl]
&&&&&&&
\circ
\ar@{-}@/^.5pc/[dd]
\ar[lllllll]
\\
{\mathrm{\phantom{deger}}}
8 &&&
&&&&&&&& &&&&&&&&& 
\bullet
&&&&&&&&&
\circ
\ar@{-}@/_.5pc/[dd]
\\
{\mathrm{\phantom{deger}}}
9 &&&
&&&&&&&& &&&&&&&&& &&&&&&&&&
\bullet
\ar@{-}[d]
&&&&&&&&
\circ
\ar@{-}[d]
\\
{\mathrm{\phantom{degr}}}
10 &&&
&&&&&&&& &&&&&&&&& &&&&&&&&&
\bullet
&&&&&&&&
\circ
\ar@{-}@/^.5pc/[dd]
\\
\\
{\mathrm{\phantom{degr}}}12 &&&
&&&&&&&& &&&&&&&&& &&&&&&&&& &&&&&&&&
\bullet
\\
{\mathrm{degree}}
 &&&
F_{-1}
&&&&&&&& 
F_0
\ar_{f_{-1}}[llllllll]
&&&&&&&&& 
F_1
\ar_{f_0}[lllllllll]
&&&&&&&&& 
F_2
\ar_{f_1}[lllllllll]
&&&&&&&&
F_3
\ar_{f_2}[llllllll]
}
$
} 

\end{proof}

\begin{theorem}
\label{fourthloops}
If $M \in \cA(1)\Mod^{(1)}$, then $\Omega^i M \hequiv M_i \tensor M$.
In particular, $\Omega^{i+4} M \hequiv \Sigma^{12} \Omega^i M$.
\end{theorem}

\begin{proof}
The modules $F_i$ have no $Q_1$ homology, while $M$ has only $Q_1$ homology.
Therefore, the $F_i \tensor M$ are $\cA(1)$-free by the K\"unneth isomorphism and
Theorem~\ref{stableiso}.  Since $M_0 \tensor M \cong M$,
the result follows from exactness of the sequence of $F_i \tensor M$.
\end{proof}


\section{Reduction from $P_1^{\tensor(n)}$ to $\Omega^n P_1$}
\label{PntoPone}

In this section we introduce the $Q_i$-localizations of $\F_2$ and
determine some of their basic properties.  As a corollary, we will
obtain the stable isomorphism type of $H^*BV$ for elementary abelian
2-groups $V$.


\begin{defn}
Let $P_1 = H^*BC_2 = (x)$, the ideal generated by $x$ in
$H^*{BC_2}_+ = \F_2[x]$.  
Let $P_0$ be 
the submodule of the Laurent series ring $L = \F_2[x,x^{-1}]$
which is nonzero in degrees  $-1$ and higher.
Let $R$ be the quotient of the unique inclusion $\eta : \F_2 \lra P_0$.
Let $\epsilon : \Sigma R \lra \F_2$ be the unique non-zero homomorphism.
\end{defn}

We represent $P_0$, $P_1$ and $R$ diagrammatically by showing the action of $Sq^1$
and $Sq^2$:

\noindent
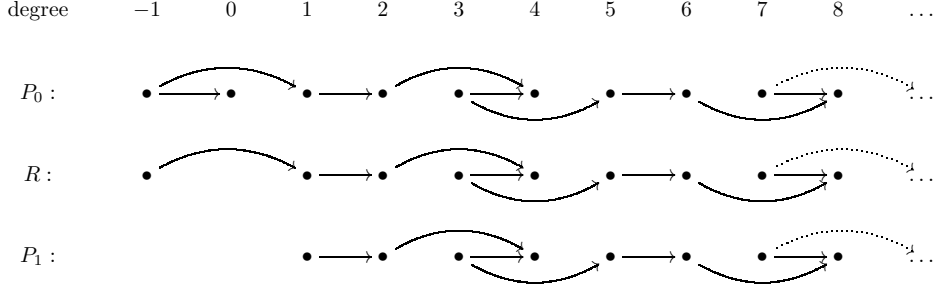
\scalebox{.82}{
$
\xymatrix{
{\mathrm{degree}}
&
-1
&
0
&
1
&
2
&
3
&
4
&
5
&
6
&
7
&
8
&
\ldots
\\
P_0:
&
\bullet 
\ar[r]
\ar@/^1pc/[rr]
&
\bullet
&
\bullet 
\ar[r]
&
\bullet 
\ar@/^1pc/[rr]
&
\bullet 
\ar[r]
\ar@/_1pc/[rr]
&
\bullet
&
\bullet 
\ar[r]
&
\bullet
\ar@/_1pc/[rr]
&
\bullet
\ar[r]
\ar@/^1pc/@{.>}[rr]
&
\bullet
&
\ldots
\\
R:
&
\bullet 
\ar@/^1pc/[rr]
&
&
\bullet 
\ar[r]
&
\bullet 
\ar@/^1pc/[rr]
&
\bullet 
\ar[r]
\ar@/_1pc/[rr]
&
\bullet
&
\bullet 
\ar[r]
&
\bullet
\ar@/_1pc/[rr]
&
\bullet
\ar[r]
\ar@/^1pc/@{.>}[rr]
&
\bullet
&
\ldots
\\
P_1:
&&&
\bullet 
\ar[r]
&
\bullet 
\ar@/^1pc/[rr]
&
\bullet 
\ar[r]
\ar@/_1pc/[rr]
&
\bullet
&
\bullet 
\ar[r]
&
\bullet
\ar@/_1pc/[rr]
&
\bullet
\ar[r]
\ar@/^1pc/@{.>}[rr]
&
\bullet
&
\ldots
\\
}
$
}  
\smallskip

We record some obvious facts using the results of the preceding section.

\begin{prop}
\label{key-sequence}
The following hold.
\begin{itemize}
\item The module $R$ is $Q_0$-local, and the
 modules $P_0$ and $P_1$ are $Q_1$-local.
\item There are short exact sequences
\[
0 \lra \Sigma P_1 \lra \Sigma R \llra{\epsilon} \F_2 \lra 0
\]
and
\[
0 \lra \F_2 \llra{\eta} P_0 \lra R \lra 0.
\]

\item $\epsilon $ is the extension cocycle for the second of these exact
sequences, giving a triangle
\[
\Sigma R \llra{\epsilon} \F_2 \llra{\eta} P_0
\]
in $\St(\cA(1)\Mod)$.
\end{itemize}
\end{prop}

\begin{proof}
All but the last item are clear from inspection.
If we let $F = R \tensor \cA(0)$, then, as in the proof of Theorem~\ref{onefold},
$F$ is $\cA(1)$-free and lies in a short exact sequence
$0\lra \Sigma R \lra F \lra R \lra 0$.
The epimorphism $F \lra R$ lifts to $P_0$,
yielding a diagram
\[
\xymatrix{
0 
\ar[r]
&
 \F_2 
\ar^{\eta}[r]
&
P_0 
\ar[r] 
&
R 
\ar[r]
&
 0
\\
0
\ar[r]
&
\Sigma R
\ar^{\epsilon}[u]
\ar[r]
&
F
\ar[r]
\ar[u]
&
R
\ar[r]
\ar@{=}[u]
&
0
}
\]
\end{proof}

We will see in Section~\ref{seclocalization} that
the map $\eta$ is the $Q_1$-localization of $\F_2$, with corresponding
$Q_1$-nullification $\epsilon$.  Dually, $\epsilon$ is the $Q_0$-colocalization
of $\F_2$ with corresponding $Q_0$-conullification $\eta$.  
As noted in the introduction, it follows that if  $I=P_0$ or $I=\Sigma R$ 
then $I$ is
idempotent, and that therefore $\Omega^n I \hequiv (\Omega I)^{\tensor(n)}$.
This underlies the argument which we now use
 to produce  minimal representatives for the tensor powers of 
$H^*BC_2$.

\begin{theorem}
\label{firstgeneral}
If $M \in \cA(1)\Mod^{(1)}$ then 
$\Omega M \hequiv \Sigma P_1 \tensor M$
and
$\eta  \tensor 1 $ is a stable equivalence
$ M \llra{\hequiv} P_0 \tensor M$.
  In particular, for $n \geq 1$,
 $P_1^{\tensor(n)}  \hequiv \Sigma^{-n} \Omega^{n} P_0$.
If $M \in \cA(1)\Mod^{(0)}$ then 
$\epsilon  \tensor 1 $ is a stable equivalence
$ \Sigma R  \tensor M \llra{\hequiv}  M$.
\end{theorem}

\begin{proof}
If $M$ is $Q_1$-local, then $ R \tensor M$ has trivial
$Q_i$-homology for both $i=0$ and $i=1$.  If $M \in \cA(1)\Mod^{(1)}$
then $ R \tensor M$ is also bounded-below, and hence free by
Theorem~\ref{stableiso}.
Tensoring the first short
exact sequence of the preceding proposition with $M$ then gives that
$\Omega M \hequiv \Sigma P_1 \tensor M$.  
Tensoring the second one with $M$ shows that $\eta \tensor 1 $ is a
 stable equivalence.

Since $P_0$ and $P_1$ are in  $\cA(1)\Mod^{(1)}$,
we have $\Omega P_0 \hequiv \Sigma P_1 \tensor P_0 \hequiv \Sigma P_1$,
 proving the $n=1$ case of the equivalence
 $P_1^{\tensor(n)}  \hequiv \Sigma^{-n} \Omega^{n} P_0$.
The remaining cases then follow immediately 
by induction:
\begin{align*}
\Omega^{n+1} P_0 = \Omega \Omega^n P_0 & \hequiv \Omega \Sigma^n P_1^{\tensor(n)} \\
& \hequiv \Sigma P_1 \tensor \Sigma^n P_1^{\tensor(n)}
\iso \Sigma^{n+1} P_1^{\tensor({n+1})}
\end{align*}
The last statement is proved dually:  since $\Sigma P_1 \tensor M$ is free
by the K\"unneth isomorphism and Theorem~\ref{stableiso},
$\epsilon \tensor 1$ is a stable equivalence.
\end{proof}

Determining minimal representatives for the tensor powers $P_1^{\tensor{n}}$
is now reduced to finding minimal representatives for the $\Omega^n P_0$.
By periodicity, we only need the first four.
The following definition will be convenient.

\begin{defn}
For $n \in \Z$, let $P_n = (\Sigma^{-n} \Omega^n P_0)^\red$.
\end{defn}

Clearly the notation is consistent with
our definitions of $P_i$, $i=0,1$.  We first record some obvious facts.

\begin{theorem}
\label{multmain} 
The modules $P_n$ are $Q_1$-local and satisfy the following equivalences.
\begin{itemize}
\item If $n\geq 1$ then $(P_1^{\tensor(n)})^{\red} = P_n$.
\item $P_{n+4} \iso \Sigma^8 P_n$,
\item $P_n \tensor P_m \hequiv P_{n+m}$, and
\item $\Omega P_n \hequiv \Sigma P_{n+1}$.
\end{itemize}
\end{theorem}

\begin{proof}
The first statement is immediate from the definition of $P_n$ and
Theorem~\ref{firstgeneral}.
Since $\Omega^{n+4}P_0 \hequiv \Sigma^{12}\Omega^n P_0$  by
Proposition~\ref{key-sequence}
 and 
Theorem~\ref{fourthloops}, we have a stable equivalence
$P_{n+4} \hequiv \Sigma^8 P_n$.  But, both sides are reduced and hence they are
isomorphic (Theorem~\ref{stableisored}).
The third and fourth statements are immediate consequences of the
first.
\end{proof}

The modules $M_i$ which appear in the sequence of 
Proposition~\ref{fourfoldperiodicity} (Definition~\ref{defMi})
all occur as submodules of the $P_i$.
See Figure~\ref{modules} for diagrammatic representations of them.



\begin{theorem}
\label{main} 
There are  short exact sequences
\[
\xymatrix{
0
\ar[r]
&
M_0
\ar[r]
&
P_0
\ar[r]
&
R
\ar[r]
&
0
\\
0 
\ar[r]
&
M_1
\ar[r]
&
P_1 
\ar[r]
&
\Sigma^4 R
\ar[r]
&
0
\\
0
\ar[r]
&
M_2
\ar[r]
&
P_2
\ar[r]
&
\Sigma^4 R
\ar[r]
&
0
\\
0
\ar[r]
&
M_3
\ar[r]
&
 P_3
\ar[r]
&
\Sigma^{4} R
\ar[r]
&
0
\\
}
\]
Each of these is the unique non-trivial extension, with
$Sq^1$ of the bottom class in the suspension of $R$ equal to the unique
element of $M_i$ of the relevant degree.
\end{theorem}

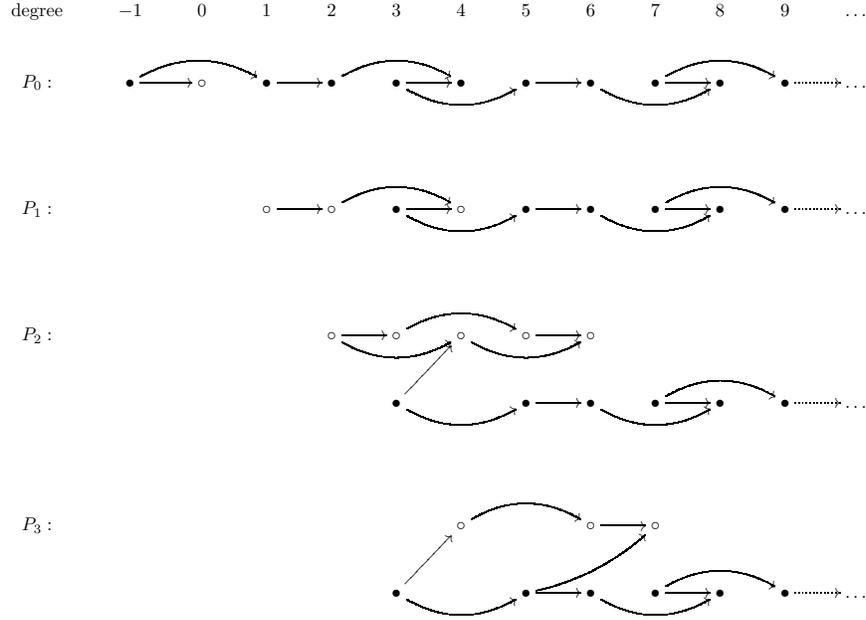
\begin{figure}
\scalebox{.7}{
$
\xymatrix{
{\mathrm{degree}}
&
-1
&
0
&
1
&
2
&
3
&
4
&
5
&
6
&
7
&
8
&
9 
&
\ldots
\\
P_0:
&
\bullet 
\ar[r]
\ar@/^1pc/[rr]
&
\circ
&
\bullet 
\ar[r]
&
\bullet 
\ar@/^1pc/[rr]
&
\bullet 
\ar[r]
\ar@/_1pc/[rr]
&
\bullet
&
\bullet 
\ar[r]
&
\bullet
\ar@/_1pc/[rr]
&
\bullet
\ar[r]
\ar@/^1pc/[rr]
&
\bullet
&
\bullet
\ar@{.>}[r]
&
\ldots
\\
&
\\
P_1:
&&&
\circ
\ar[r]
&
\circ
\ar@/^1pc/[rr]
&
\bullet 
\ar[r]
\ar@/_1pc/[rr]
&
\circ
&
\bullet 
\ar[r]
&
\bullet
\ar@/_1pc/[rr]
&
\bullet
\ar[r]
\ar@/^1pc/[rr]
&
\bullet
&
\bullet
\ar@{.>}[r]
&
\ldots
\\
&
\\
P_2:
&&&&
\circ
\ar[r]
\ar@/_1pc/[rr]
&
\circ
\ar@/^1pc/[rr]
&
\circ
\ar@/_1pc/[rr]
&
\circ
\ar[r]
&
\circ
&
\\
&&&&&
\bullet
\ar[ru]
\ar@/_1pc/[rr]
&&
\bullet
\ar[r]
&
\bullet
\ar@/_1pc/[rr]
&
\bullet
\ar@/^1pc/[rr]
\ar[r]
&
\bullet
&
\bullet
\ar@{.>}[r]
&
\ldots
\\
&
\\
P_3:
&&&&&&
\circ
\ar@/^1pc/[rr]
&&
\circ
\ar[r]
&
\circ
&
\\
&&&&&
\bullet
\ar[ru]
\ar@/_1pc/[rr]
&&
\bullet
\ar@/_.5pc/[rru]
\ar[r]
&
\bullet
\ar@/_1pc/[rr]
&
\bullet
\ar@/^1pc/[rr]
\ar[r]
&
\bullet
&
\bullet
\ar@{.>}[r]
&
\ldots
\\
}
$
} 
\caption{The modules $P_n$, $0 \leq n \leq 3$.  The submodules $M_n$
are indicated by open dots ($\circ$).
\label{modules}}
\end{figure}


\begin{proof}
The first short exact sequence is a restatement of the last short exact sequence in
Proposition~\ref{key-sequence}.
Next, the submodule of $P_1$ generated by the bottom class is $M_1$ and
the quotient by it is $\Sigma^4 R$.  This gives 
\begin{equation}
\label{seq2}
0 \lra  M_1 \lra P_1 \lra \Sigma^{4}R \lra 0,
\end{equation}
the second of our claimed short exact sequences.
Taking minimal free modules mapping onto the three modules in~(\ref{seq2}) and applying
the snake lemma produces the suspension of the next of our claimed sequences.
\[
\xymatrix{
0 \ar[r]
&
M_1 \ar[r]
&
P_1 \ar[r]
&
\Sigma^4 R \ar[r]
&
0
\\
0 \ar[r]
&
\Sigma \cA(1) \ar[r] \ar[u]
&
F_1 \ar[r] \ar[u]
&
\Sigma^4 F \ar[r] \ar[u]
&
0
\\
0 \ar[r]
&
\Sigma M_2 \ar[r] \ar[u]
&
\Sigma P_2 \ar[r] \ar[u]
&
\Sigma^5 R \ar[r] \ar[u]
&
0
\\
}
\]
Here $F_1 = \Sigma \cA(1) \oplus \Sigma^4 F$ and $F$ is the free module used in
the proof of Proposition~\ref{key-sequence}.
It is easy to check that the
top and bottom rows in the preceding diagram are each the unique non-trivial
extension.

Applying this procedure again, we get
\[
\xymatrix{
0 \ar[r]
&
M_2 \ar[r] 
&
P_2 \ar[r] 
&
\Sigma^4 R \ar[r] 
&
0
\\
0 \ar[r]
&
\Sigma^2 \cA(1) \ar[r] \ar[u]
&
F_2 \ar[r] \ar[u]
&
\Sigma^4 F \ar[r] \ar[u]
&
0
\\
0 \ar[r]
&
\Sigma^{9}DM_1 \ar[r] \ar[u]
&
\Sigma P_3 \ar[r] \ar[u]
&
\Sigma^5 R \ar[r] \ar[u]
&
0
\\
}
\]
Here $F_2 = \Sigma^2 \cA(1) \oplus \Sigma^4 F$
and
$DM = \Hom(M,\F_2)$ is the dual of $M$.
Removing one suspension gives the last of our short exact sequences.
\end{proof}

\begin{remark}
\label{error}
In \cite{Ossa}, Lemma 2 asserts that $P_1 \tensor P_1$ is stably equivalent to
$\Sigma^2 P_1$ rather than $\Sigma^{-1}\Omega P_1$.  
These are the same
in the category of $E(1)$-modules, but
not in the category of 
$\cA(1)$-modules.
These modules differ by one copy of $E(1)$.
This also makes Proposition 2 there false, both in identifying the degrees of the
Eilenberg-MacLane summands, and in identifying the complement to them.
See Corollary~\ref{PP} for the correct statement.
\end{remark}

\section{$Q_i$-local $A(1)$-modules}
\label{seclocalization}




Again let $B=E(1)$ or $\cA(1)$.  We now consider the
two Margolis localizations (at $Q_0$ and at $Q_1$) of $B\Mod$.

\begin{defn} 
\label{Lidefn}
Let  $\epsilon : \Sigma R \lra \F_2$ and $\eta : \F_2 \lra P_0$ be the
unique non-zero homomorphisms.  Define functors
$\LL_i : B\Mod^b \lra B\Mod^{(i)}$  and natural transformations
 $\eta_M : M \lra \LL_1 M$ 
and
 $\epsilon_M : \LL_0 M \lra M$  by
\[
\xymatrix{
\Sigma R \tensor M
\ar@{=}[d]
\ar^{\epsilon \tensor 1}[r]
&
\F_2 \tensor M
\ar^{\eta \tensor 1}[r]
\ar^{\iso}[d]
&
P_0 \tensor M
\ar@{=}[d]
\\
\LL_0 M
\ar_{\epsilon_M}[r]
&
M
\ar_{\eta_M}[r]
&
\LL_1 M
\\
}
\]
\end{defn}

The functors they induce on stable module categories
are idempotent, orthogonal,  semi-ring homomorphisms.
We make these statements precise as follows.

\begin{theorem}
\label{Libasicprops}
$\LL_i M$ is $Q_i$-local.
$\LL_0$ and $\LL_1$ are exact and additive, and preserve tensor products
up to stable equivalence.
\begin{enumerate}
\item $\LL_0 \LL_1 M \hequiv 0 \hequiv \LL_1 \LL_0 M$.
\item $\epsilon_M$ induces an isomorphism of $Q_0$ homology.
\item $\eta_M$ induces an isomorphism of $Q_1$ homology.
\item 
If $M\in \cA(1)\Mod^{(0)}$ then
$\epsilon_M$ is a stable equivalence 
and $\LL_1 M \hequiv 0$.
\item 
If $M\in \cA(1)\Mod^{(1)}$  then
$\eta_M$ is a stable equivalence 
and $\LL_0 M \hequiv 0$.
\end{enumerate}
\end{theorem}

\begin{proof}
That $\LL_i M$ is $Q_i$-local is immediate from the K\"unneth theorem
for $Q_j$ homology and Proposition~\ref{key-sequence}.
It is a general fact that tensor product is exact and preserves direct sums.  
Applying Theorem~\ref{firstgeneral} to $M=R$ and $M=P_0$,
we find that $\Sigma R \tensor \Sigma R \hequiv \Sigma R$ and
$P_0 \tensor P_0 \hequiv P_0$.
Preservation of tensor products then follows by associativity
and this idempotence.
Statement (1) follows from the fact that
$\Sigma R \tensor P_0$ is free by Proposition~\ref{key-sequence},
the K\"unneth theorem, and
Theorem~\ref{stableiso}.  Then (2) and (3) 
follow from the K\"unneth  theorem for $Q_i$ homology
and the case $M = \F_2$.
Finally, (4) and (5) are then immediate by the theorem of
Adams and Margolis (Theorem~\ref{stableiso}).
\end{proof}

Here is a  more precise form of idempotence.

\begin{theorem}
The $\LL_i$ are stably idempotent.  In particular, the following hold.
\begin{enumerate}
\item
$\LL_0 \epsilon_M$,  $\epsilon_{\LL_0 M}$,
$\LL_1 \eta_M$, and $\eta_{\LL_1 M}$ are  stable equivalences.
\item
$\LL_0 \epsilon_M \hequiv \epsilon_{\LL_0 M}$ and
$\LL_1 \eta_M \hequiv \eta_{\LL_1 M}$.
\item
$\LL_0 \epsilon_M$ and $ \epsilon_{\LL_0 M}$ 
are  not equal, but are
coequalized by $\epsilon_M$.
\item
$\LL_1 \eta_M$ and $ \eta_{\LL_1 M}$ are 
not equal, but are
equalized by $\eta_M$.
\end{enumerate}
\end{theorem}


\begin{proof}
Statement (1) is immediate from the preceding Theorem.
Statements (3) and (4) are elementary calculations:
$\epsilon \tensor 1$ and $1 \tensor \epsilon$ are coequalized
by $\epsilon$, while
$\eta \tensor 1$ and $1 \tensor \eta$ are equalized
by $\eta$.
To show the stable equivalences in (2), 
it suffices to treat the case $M = \F_2$.  For this,
we use Proposition~\ref{stableisored}.  
Since $(\Sigma R \tensor \Sigma R)^{\red} \iso \Sigma R$,
we need a stable equivalence $\Sigma R \lra \Sigma R \tensor \Sigma R$
which equalizes $\epsilon\tensor 1 $ and $1 \tensor \epsilon$.
We define such an $\cA(1)$ homomorphism by
\[
i(\Sigma x^n) = \sum_{i+j = n-1} 
\Sigma x^i \tensor \Sigma x^j  
\]
where we treat
$\Sigma x^0 $ as zero, and let $i$ and $j$ range over integers
$\geq -1$.
It is immediate that 
$(\epsilon \tensor 1)i = (1 \tensor \epsilon)i$ so that
$\epsilon \tensor 1 \hequiv 1 \tensor \epsilon$.

Dually, for the stable equivalence between $\eta \tensor 1$ and
$1 \tensor \eta$,  we
observe that they are coequalized by the stable equivalence
$P_0 \tensor P_0 \lra P_0$ given by
\[
x^i \tensor x^j \mapsto
\left\{ \begin{array}{ll}
0 & i \congruent -1 \pmod{4}\,\,\,\, \,\,{\mathrm{ and }}\,\,\,\,\,\, j \congruent -1 \pmod{2} \\
x^{i+j} & {\mathrm{otherwise.}}\\
\end{array} \right.
\]
\end{proof}

\begin{prop}
Algebraic loops commute with the $\LL_i$:
$\Omega \LL_i M \hequiv \LL_i \Omega M$.  In addition,
\begin{enumerate}
\item $\Omega \LL_0 M \hequiv \Sigma \LL_0 M$
\item $\Omega^i \LL_1 M \hequiv \Sigma^i P_i \tensor M $
\item $ \Omega^4 \LL_1 M \hequiv \LL_1 \Omega^4 M \hequiv \Sigma^{12} \LL_1 M$
\end{enumerate}
\end{prop}

\begin{proof}
Since tensoring a $B$-module with a free $B$-module gives a free $B$-module,
we have
\[
\Omega(M \tensor N) \hequiv (\Omega M) \tensor N \hequiv M \tensor (\Omega N).
\]
Then $\Omega \LL_0 M = \Omega (\Sigma R \tensor M) \hequiv \Sigma^2 R \tensor M$,
proving (1).  Similarly (2) follows because
$\Omega^i \LL_1 M = \Omega^i (P_0 \tensor M) \hequiv (\Omega^i P_0) \tensor M
\hequiv \Sigma^i P_i \tensor M $ by Theorem~\ref{multmain}.
Then (3) follows since $\Sigma^4 P_4 = \Sigma^{12} P_0$.
\end{proof}


\begin{prop}
\label{Lsubiproperties}
There are short exact sequences
\[
0 \lra M \llra{\eta_M} \LL_1(M) \lra \Sigma^{-1}\LL_0(M) \lra 0
\]
and
\[
0 \lra \Omega \LL_1(M) \lra \LL_0(M) \llra{\epsilon_M} M \lra 0.
\]
\end{prop}

\begin{proof}
These follow from the short exact sequences of modules
\[
0 \lra \F_2 \llra{\eta} P_0 \lra R \lra 0
\]
and
\[
0 \lra \Sigma P_{1} \lra \Sigma R \llra{\epsilon} \F_2 \lra 0.
\]
\end{proof}


\begin{theorem}
\label{uniquetriangle}
Each $M \in St(B\Mod^{b})$  sits in a unique triangle
$M_0 \lra M \lra M_1$ with $M_i \in B\Mod^{(i)}$.  
Therefore, a $B$-module in $B\Mod^b$ is
uniquely 
determined, up to stable equivalence,
by a triple $(M_0,M_1,e(M))$, where $M_i\in B\Mod^{(i)}$
 and $e(M) \in \Ext_B^{1,0}(M_1,M_0)$.
\end{theorem}

\begin{proof}
The diagram
\[
\xymatrix{
\LL_0 M_0
\ar^{\hequiv}[r]
\ar^{\hequiv}[d]
&
M_0
\ar[r]
\ar[d]
&
\LL_1 M_0 \hequiv 0
\ar[d]
\\
\LL_0 M
\ar^{\epsilon_M}[r]
\ar[d]
&
M
\ar^{\eta_M}[r]
\ar[d]
&
\LL_1 M
\ar^{\hequiv}[d]
\\
0 \hequiv \LL_0 M_1
\ar[r]
&
M_1
\ar^{\hequiv}[r]
&
\LL_1 M_1
\\
}
\]
shows that the triangle $M_0 \lra M \lra M_1$ is equivalent to the
canonical one, $\LL_0 M \lra M \lra \LL_1 M$.
\end{proof}

\begin{remark}
\label{generalstableequivalence}
Finally, it is clear that we can extend these definitions to
all $B$-modules.  
The fundamental triangle 
\[
\LL_0 M \llra{\epsilon_M} M \llra{\eta_M} \LL_1 M
\]
 then
implies that a homomorphism $f:M\lra N$ 
in $B\Mod$
is a stable equivalence 
iff both $\LL_0(f)$ and $\LL_1(f)$ are stable equivalences.
It shows, in particular, that $M$ is free iff both
$\LL_0 M$ and $\LL_1 M$ are free.

This criterion for equivalence is the same as 
that of Adams and Margolis for bounded below 
modules, but holds in full generality.
The example of Section~\ref{final} shows that this is a genuine generalization.
\end{remark}


\section{$\Pic$ and $\Pic^{(k)}$}
\label{sec:Pic}

Again let $B$ be either $E(1)$ or $\cA(1)$ and $\St(\cC)$ the stable
category of a subcategory $\cC$ of $B\Mod$ (Definition~\ref{defred}).
Since $\F_2$ is the unit for tensor product in $B\Mod$, its localizations
$\LL_i\F_2$ are the units for tensor product in the local subcategories.

\begin{prop}
\label{units}
$\Sigma R$ is the unit for tensor product in $\St(B\Mod^{(0)})$
and 
$P_0$ is the unit for tensor product in $\St(B\Mod^{(1)})$.
The stable equivalence classes of modules  $M\in St(B\Mod^{(k)})$ 
or $\St(B\Mod^b)$ form a (possibly big) semi-ring with unit
under direct sum and tensor product.
\qqed
\end{prop}

The Picard groups are the multiplicative groups in these semi-rings.

\begin{defn}
\label{defPic}
Let 
\begin{itemize}
\item $\widetilde{\Pic}(B) =
\left({\mathrm{Obj}}(\St(B\Mod^b))/{\!\!\hequiv}\right)^{\times}$
and 
\item $\Pic^{(k)}(B) =
\left({\mathrm{Obj}}(\St(B\Mod^{(k)}))/{\!\!\hequiv}\right)^{\times}$.
\end{itemize}
Let $\Pic(B)$ be the
subgroup of $\widetilde{\Pic}(B)$ whose elements are represented by
finitely generated modules.
\end{defn}

\begin{remark}
Of these, only $\Pic(B)$ is clearly a set.  We will show, by explicitly
calculating them, that $\Pic^{(0)}(B)$
and   $\Pic^{(1)}(B)$ are sets.
It would be interesting to know whether $\Pic(B) = \widetilde{\Pic}(B)$, and,
if not, how much larger $\widetilde{\Pic}(B)$ is.
\end{remark}

Adams and Priddy 
characterize the elements in $\Pic(B)$.
It is pertinent to recall that $H(M,Q_k)$ depends only upon the stable
isomorphism type of $M$.

\begin{lemma}\cite[Lemma 3.5]{BSO}
$M \in \Pic(B)$ iff
each $H(M,Q_k)$ is one dimensional.
\end{lemma}

Adams and Priddy remark that, if one drops the hypothesis of finite generation, then 
having $H(M,Q_k)$ one dimensional for each $k$ no longer implies that
$M$ is invertible.  The module $P_0 \dirsum \Sigma R$ is an example.
The other direction does hold in general, though.

\begin{lemma}
\label{onedim}
 If $M \in \Pic^{(k)}(B)$ then
$H(M,Q_k)$ is one dimensional.
\end{lemma}

The converse, Corollary~\ref{onedimconverse},
will follow from the calculations of $\Pic^{(k)}$,
since those
calculations will show that if $M \in B\Mod^{(k)}$ and $H(M,Q_k)$ is one dimensional,
then $M$ is stably isomorphic to an invertible module. 

After characterizing the invertible $B$-modules,
Adams and Priddy go on to compute $\Pic(E(1))$ and $\Pic(\cA(1))$.

\begin{theorem}\cite[Theorem 3.6]{BSO}
\label{AdamsandPriddyE}
$\Pic(E(1)) = \Z \dirsum \Z$, generated by $\Sigma \F_2$ and the
augmentation ideal $\Omega \F_2 = \Ker(E(1) \lra \F_2)$.
\end{theorem}

\begin{theorem}\cite[Theorem 3.7]{BSO}
\label{AdamsandPriddy}
$\Pic(\cA(1)) = \Z \dirsum \Z \dirsum \Z/(2)$, generated by $\Sigma \F_2$, the
augmentation ideal $\Omega \F_2 = \Ker(\cA(1) \lra \F_2)$, and $\Sigma^{-4} M_2$.
\end{theorem}

The module $J = \Sigma^{-4} M_2$ is known as the `joker' for its role as a torsion
element in $\Pic(\cA(1))$ and for the resemblance of its diagrammatic depiction
(Figure~\ref{modules}) to a traditional jester's hat.

We now turn to the determination of the local $\Pic$ groups.
In his thesis (\cite{Yu}), Cherng-Yih Yu computed $\Pic^{(1)}$ for
both $E(1)$ and $\cA(1)$.    His calculation of $\Pic^{(1)}(E(1))$
is easy, and we give it now.
His calculation of $\Pic^{(1)}(\cA(1))$ is quite
complicated and computational.
In the next section, we give a simpler and more straightforward
calculation of it.
Following that, we compute
$\Pic^{(0)}$ for both $E(1)$ and $\cA(1)$.

Recall that, as an $E(1)$-module,
$P_i \simeq \Sigma^{2i} P_0$ (see Remark~\ref{error}).

\begin{theorem} \cite[Lemma 2.5]{Yu}
\label{Yuslemma}
If $M \in E(1)\Mod^{(1)}$ and $H(M,Q_1) = \Sigma^s \F_2$ then
$M \simeq \Sigma^{s}P_0$.  Therefore,
$\Pic^{(1)}(E(1)) = \{\Sigma^i P_0\} \iso \Z$.
\end{theorem}
 
\begin{proof}
By suspending appropriately, we may assume that
$M \in E(1)\Mod^{(1)}$ and $H(M,Q_1) = \F_2$.
We may also assume that $M$ is reduced, i.e., $Q_1 Q_0 = 0$.
Let $0 \neq \langle[x]\rangle = H(M,Q_1)$, so that $Q_1(x) = 0$ and $x\notin \Image(Q_1)$.
There are two possibilities:
\begin{enumerate}
\item $Sq^1 x \neq 0$
\item $Sq^1 x = 0$
\end{enumerate}
In the first case, $Q_1 Sq^1 x = 0$ because $M$ is reduced, so $Sq^1 x = Q_1 x_1$
for some $x_1$.  (This because $[x]$ is the only nonzero $Q_1$ homology class of $M$.)
Again,  $M$  reduced implies that
$x_1 \notin \Image(Sq^1)$,  so that $Sq^1 x_1 \neq 0$.  
Since $Q_1 Sq^1 x_1 = 0$, we have $Sq^1 x_1 = Q_1 x_2$ for
some $x_2$.  Continuing in this way,
it follows by induction that
$M$ is not bounded-below, contrary to our assumption.

Therefore, we must have $Sq^1x = 0$.
Then $x= Sq^1 x_0$ for some $x_0$ and $x_0 \notin \Image(Q_1)$ because $M$
is reduced.  Hence $Q_1x_0 \neq 0$.  Again, the fact that $M$ is reduced implies
that $Q_1x_0 = Sq^1x_1$ for some $x_1$.  For induction, we may suppose that
we have found a sequence of elements $x_i$ such that 
$Q_1 x_{i-1} = Sq^1x_i \neq 0$, for $0 \leq i \leq n$.  
Then, since $M$ is reduced, $x_n\notin \Image(Q_1)$, so $Q_1x_n \neq 0$
and there must be $x_{n+1}$ such that $Q_1x_{n+1} = Sq^1 x_n$.

The submodule of $M$ generated by the $x_i$ is isomorphic to $P_0$ and the inclusion
$P_0 \lra M$ induces an isomorphism of $Q_k$ homologies, hence is a stable isomorphism
by Theorem~\ref{stableiso}.

Now suppose that $M\in \Pic^{(1)}(E(1))$.  By Lemma~\ref{onedim},
$H(M,Q_1) = \Sigma^s\F_2$ for some $s$,
and therefore $M \simeq \Sigma^s P_0$.
Finally, observe that the $\Sigma^i P_0$ are all distinct because 
$H(\Sigma^i P_0,Q_1) = \Sigma^i \F_2$.  
\end{proof}

Here is the result for $\cA(1)$.

\begin{theorem} \cite[Theorem 2.1]{Yu}
\label{Yustheorem}
If $M \in \cA(1)\Mod^{(1)}$ and $H(M,Q_1) = \Sigma^a \F_2$ then
$M \simeq \Sigma^{a-2b}P_b$ for some $b$.  Therefore,
$\Pic^{(1)}(\cA(1)) = \{\Sigma^i P_n\} \iso \Z \dirsum \Z/(4)$ with 
$(a,b) \in \Z \dirsum \Z/(4)$ corresponding to $\Sigma^{a-2b} P_b$.
\end{theorem}

\begin{proof} 
The first statement is the key technical result, and will be
given as
Theorem~\ref{thm:proofofYustheorem} in the next section.  For the
remainder,
suppose that $M\in \Pic^{(1)}(\cA(1))$.  By Lemma~\ref{onedim},
the first statement applies to show that $M = \Sigma^i P_n$ for some
$i$ and $n$.  The multiplicative structure then follows from Theorem~\ref{multmain}.
\end{proof}

\section{The proof of Yu's Theorem}
\label{sec:proofofYustheorem}

\begin{theorem}
\label{thm:proofofYustheorem}
If $M \in \cA(1)\Mod^{(1)}$ and $H(M,Q_1) = \Sigma^a \F_2$ then
$M \simeq \Sigma^{a-2b}P_b$ for some $b$. 
\end{theorem}

\begin{proof}
We will assume that $a=0$.  We may also assume that $M$ is reduced:
$Sq^2Sq^2Sq^2$ acts as 0 on $M$.  By Theorem~\ref{Yuslemma},
as an $E(1)$-module we have
\[
M|_{E(1)} \iso P_0 \dirsum (E(1) \tensor V)
\]
for some bounded-below graded $\F_2$-vector space $V$.  
Recall that $\cA(1)$ is generated by $E(1)$ and $Sq^2$.
Therefore, to describe
$M$ as an $\cA(1)$-module, it remains to specify the action of $Sq^2$ on $M$
in a manner consistent with its structure as an $E(1)$-module.
This is given by the following cocycle data.  First, we have
\begin{enumerate}
\item a linear functional $s: V \lra \F_2$, and
\item linear transformations
\begin{enumerate}
\item $u : V_i \lra V_{i+1}$,
\item $v : V_i \lra V_{i-1}$, and
\item $w : V_i \lra V_{i-2}$,
\end{enumerate}
\end{enumerate}
such that,
for all $y \in V$,
\begin{equation}
\label{sqtwoy}
Sq^2 y = s(y) x_{t(y)} + Sq^1 u(y) + Q_1 v(y) + Sq^1 Q_1 w(y).
\end{equation}
Here, $ t(y) = 2 + |y|$ and $x_t$ is the nonzero element of $P_0$ in degree $t$
when $t \geq -1$.  If $t < -1$ we take $x_t$ to be 0, though we will see shortly that
this cannot occur.
There can be no term in $V$ itself since $M$ is reduced.  

Similarly, we have sequences indexed on the integers $i \geq -1$:
\begin{enumerate}
\item $a_i \in \F_2$,
\item $b_i \in V_{i+1}$,
\item $c_i \in V_{i-1}$, and
\item $d_i \in V_{i-2}$,
\end{enumerate}
such  that
\begin{equation}
\label{sqtwoxi}
Sq^2 x_i = a_i x_{i+2} + Sq^1 b_i + Q_1 c_i + Sq^1Q_1 d_i.
\end{equation}
Again, there can be no term in $V$ itself since $M$ is reduced.  
It will be convenient to declare $a_i$, $b_i$, $c_i$ and $d_i$ to be $0$ when 
$i < -1$.

Our main tools will be the direct sum decomposition (over $\F_2$)
\[
M \iso P_0 \dirsum V \dirsum Sq^1 V \dirsum Q_1 V \dirsum Sq^1 Q_1 V
\]
and the observation that the elements of $E(1)$ act monomorphically on $V$.

We now need a series of Lemmas.

\renewcommand{\qedsymbol}{}
\end{proof}

First, consider the consequences of the relation $Q_1 = Sq^1 Sq^2 + Sq^2 Sq^1$
on $P_0$.

\begin{lemma}  
\label{xconditions}
The action of $Sq^2$ on $P_0$ satisfies the following relations:
\begin{enumerate}
\item $a_{2i-1} + a_{2i} = 1$ for $i \geq 0$,
\item $b_{2i} = 0$,
\item $c_{2i} = 0$, and
\item $d_{2i} = c_{2i-1}$.
\end{enumerate}
\end{lemma}

\begin{proof}
From equation~(\ref{sqtwoxi}) we have
\begin{align*}
Sq^1 Sq^2 x_i & =  i a_i x_{i+3} + Sq^1 Q_1 c_i \\
Sq^2 Sq^1 x_i & =  i(a_{i+1} x_{i+3} + Sq^1 b_{i+1} + 
  Q_1 c_{i+1}+
  Sq^1 Q_1 d_{i+1} )
\end{align*}

Since $ Q_1 x_i  =  i x_{i+3}$, comparing coefficients of the direct sum
decomposition of $M$ gives
\[
i(1+a_i+a_{i+1}) = 0
\]
for $i \geq -1$, together with
\begin{align*}
i b_{i+1} &= 0 \\
i c_{i+1} &= 0 \\
c_i + i d_{i+1} &=0
\end{align*}
for all $i$. This implies the relations given.
\end{proof}

Next, we consider the action of $Sq^2$ on the free $E(1)$-module generated
by $V$.

\begin{lemma}
\label{sqtwoaction}
Let $y \in V$ and $t = t(y)$.  Then
\begin{enumerate}
\item
$ Sq^1 Sq^2 y = s(y) t(y) x_{t+1} + Sq^1 Q_1 v(y)$,
\item
$Sq^2 Sq^1 y  = Q_1 y +  s(y) t(y) x_{t+1} + Sq^1 Q_1 v(y)$, and
\item
$Sq^2 Q_1 y 
= s(y) t(y) \left( a_{t+1} x_{t+3} + Sq^1 b_{t+1} + Q_1 c_{t+1} + Sq^1 Q_1 d_{t+1}
  \right)$.
\end{enumerate}
\end{lemma}

\begin{proof}
Applying $Sq^1$ to equation~(\ref{sqtwoy}) gives
\[
Sq^1 Sq^2 y = s(y) t(y) x_{t+1} + Sq^1 Q_1 v(y).
\]
We must then have
\begin{align*}
Sq^2 Sq^1 y &= Q_1 y + Sq^1 Sq^2 y\\
 & = Q_1 y +  s(y) t(y) x_{t+1} + Sq^1 Q_1 v(y)
\end{align*}
and
\begin{align*}
Sq^2 Q_1 y &= Sq^2 Sq^1 Sq^2 y = 
     Sq^2\left( s(y) t(y) x_{t+1}  + Sq^1 Q_1 v(y) \right)  \\
&= s(y)t(y) Sq^2 x_{t+1}  \\
&= s(y) t(y) \left( a_{t+1} x_{t+3} + Sq^1 b_{t+1} + Q_1 c_{t+1} + Sq^1 Q_1 d_{t+1}
  \right)
\end{align*}
where we have used that 
$Sq^2 Sq^2 Sq^1 =0$ and that
$Sq^2 Sq^1 Q_1$ acts trivially since $M$ is reduced.
\end{proof}

Now we turn to the consequences of the relation $Sq^2Sq^2 = Sq^1 Q_1$ on $V$.
These give stringent restrictions on the vector space $V$.

\begin{lemma}
\label{Vconditions}
Each $V_i$ is at most one-dimensional.  In addition, we have the following.
\begin{enumerate}
\item $V_i = 0$ if $i < -3$.
\item $V_{2i-2}$ is spanned by $d_{2i}$.  If it is nonzero, then
 $a_{2i} = 0$, and
\[
s(d_{2i}) + s(v(d_{2i})) = 1.
\]
\item $V_{2i-1}$ is spanned by $d_{2i+1} + v(c_{2i+1})$.  If it is nonzero,
then $a_{2i+1} = 0$, $b_{2i+1} = 0$, 
\[
s(d_{2i+1} + v(c_{2i+1})) = 1,
\] 
and
\[
c_{2i+1} = u(d_{2i+1} + v(c_{2i+1})).
\]
\end{enumerate}
\end{lemma}

\begin{proof}
Applying the preceding Lemma to $y \in V_{t-2}$, we find

\begin{align*}
Sq^1 Q_1 y & =
   Sq^2 Sq^2 y =
     Sq^2 \left( s(y) x_t + Sq^1 u(y) + Q_1 v(y) + Sq^1 Q_1 w(y) \right) \\
 &= s(y) \left( a_t x_{t+2} + Sq^1 b_t + Q_1 c_t + Sq^1 Q_1 d_t \right) \\
&\phantom{=} + Q_1 u(y) + s(u(y)) (1+t) x_{t+2} + Sq^1 Q_1 v(u(y)) \\
&\phantom{=} + s(v(y)) (1+t) \left(
      a_t x_{t+2} + Sq^1 b_t + Q_1 c_t + Sq^1 Q_1 d_t \right).
\end{align*}

\noindent
where we let t = t(y).
Separating terms from distinct summands, we get
\begin{align*}
0 & = s(y) a_t + s(u(y))(1+t) + s(v(y)) (1+t) a_t \\
0 &= \left( s(y) + s(v(y)) (1+t) \right) b_t  \\
0 &= \left( s(y) + s(v(y)) (1+t) \right) c_t + u(y)  \\
y &=  \left( s(y) + s(v(y)) (1+t) \right) d_t + v(u(y))  
\end{align*}

If $t=2i$, then putting $c_{2i} = 0$ in the third equation implies that
$u : V_{2i-2} \lra V_{2i-1}$ is 0.  Then $v(u(y)) = 0$ as well,
so that the last equation gives
\[
y = \left( s(y) + s(v(y)) \right) d_{2i}.
\]
Hence $V_{2i-2}$ is at most one dimensional, spanned by $d_{2i} $.
If $d_{2i} \neq 0$ then, letting $y=d_{2i}$ in this equation gives
\[
 s(d_{2i}) + s(v(d_{2i})) = 1.
\]
The first of our 4 summands then gives 
\begin{align*}
0 & = \left( s(d_{2i}) + s(v(d_{2i})) \right) a_{2i} + s(u(d_{2i})) \\
 & = a_{2i} + s(0) \\
 & = a_{2i}.
\end{align*}

In the other parity, $t=2i+1$, so that $|y| = 2i-1$, we get
\begin{align*}
0 & = s(y) a_{2i+1} \\
0 &=  s(y)  b_{2i+1}  \\
u(y) &=  s(y)  c_{2i+1}  \\
y &=   s(y)  d_{2i+1} + v(u(y))   \\
 &=   s(y) \left( d_{2i+1} + v(c_{2i+1})  \right).
\end{align*}

We again find that $V_{2i-1}$ is at most one dimensional,
spanned now by $d_{2i+1} + v(c_{2i+1})$.  
If this is non-zero, then
letting $y = d_{2i+1} + v(c_{2i+1})$ in the last equation gives
$s(y) = 1$, from which it follows that $a_{2i+1}=0$, $b_{2i+1}=0$,
and $c_{2i+1} = u(d_{2i+1} + v(c_{2i+1}))$.

Since the lowest degree nonzero $d_i$ is $d_{-1}$, this gives 
$V_i = 0$ for $i < -3$.
\end{proof}

The action of $Sq^2 Sq^2 = Sq^1 Q_1$ on $P_0$ is already determined
by the $E(1)$-module structure of $M$.  This
eliminates most of the possibilities left open by the 
preceding Lemma.  We handle the even and odd degree cases
separately because their proofs are somewhat different.

\begin{lemma}
If $V_{2i-2} \neq 0$ then $2i-2=-2$.
\end{lemma}

\begin{proof}
Lemma~\ref{Vconditions} implies that if
 $V_{2i-2} \neq 0$ then $V_{2i-2} = \langle d_{2i}\rangle$ and 
$a_{2i} = 0$.
Lemma~\ref{xconditions} then gives $a_{2i-1} = 1$, and we have
\[
Sq^2 x_{2i-1} = x_{2i+1} + Sq^1 b_{2i-1} + Q_1 d_{2i} + Sq^1Q_1 d_{2i-1}
\]
and
\[
Sq^2 x_{2i} = Sq^1Q_1 d_{2i}.
\]

If $2i-2 \neq -2$
then $2i-2 \geq 0$ by Lemma~\ref{Vconditions}, and thus
$2i-3 \geq -1$.  Then we have
\begin{align*}
0 &= Sq^2Sq^2 x_{2i-3} \\
 &= Sq^2(a_{2i-3} x_{2i-1} + Sq^1 b_{2i-3} + Q_1 c_{2i-3} + Sq^1Q_1 d_{2i-3}) \\
 &= a_{2i-3}\left( x_{2i+1} + Sq^1 b_{2i-1} + Q_1 d_{2i} + Sq^1Q_1 d_{2i-1}\right) \\
&\phantom{=} + Q_1 b_{2i-3} + Sq^1 Q_1 v(b_{2i-3}) \\
&\phantom{=} + 0 \\
&= a_{2i-3} x_{2i+1} + Sq^1 b_{2i-1} +
  Q_1(b_{2i-3} + d_{2i}) +
  Sq^1 Q_1(b_{2i-3} + d_{2i}).
\end{align*}

Hence, $a_{2i-3} = 0$, so by Lemma~\ref{xconditions}, $a_{2i-2} = 1$.
We therefore have
\begin{align*}
0 &= Sq^2Sq^2 x_{2i-2} \\
 &= Sq^2( x_{2i} + Sq^1Q_1 d_{2i-2}) \\
 &= Sq^1Q_1 d_{2i} \\
\end{align*}
which is a contradiction.
\end{proof}

\begin{lemma}
If $V_{2i-1} \neq 0$ then $2i-1=-3$.
\end{lemma}

\begin{proof}
Lemma~\ref{Vconditions} implies that
if $V_{2i-1} \neq 0$ and $2i-1 \neq -3$ then $2i-1 \geq -1$.  Also,
$d_{2i+1}+v(c_{2i+1}) \neq 0$, $a_{2i+1}=0$ and $b_{2i+1}=0$.  By
Lemma~\ref{xconditions}, it follows that $a_{2i+2} = 1$.  Then
\[
Sq^2 x_{2i+1} = Q_1 c_{2i+1} + Sq^1Q_1 d_{2i+1}
\]
and
\[
Sq^2 x_{2i+2} = x_{2i+4} + Sq^1 Q_1 c_{2i+1}.
\]
(Recall from Lemma~\ref{xconditions} that $c_{2i+1} = d_{2i+2}$.)
Then,
\begin{align*}
0 &= Sq^2Sq^2 x_{2i} \\
 &= Sq^2(a_{2i} x_{2i+2} +  Sq^1Q_1 d_{2i}) \\
 &= a_{2i}\left( x_{2i+4}  + Sq^1Q_1 c_{2i+1}\right) 
\end{align*}

Hence, $a_{2i} = 0$, so by Lemma~\ref{xconditions}, $a_{2i-1} = 1$.
We therefore have
\begin{align*}
0 &= Sq^2Sq^2 x_{2i-1} \\
 &= Sq^2( x_{2i+1} + Sq^1 b_{2i-1} + Q_1 c_{2i-1} + Sq^1Q_1 d_{2i-1}) \\
 &= Q_1 c_{2i+1} + Sq^1Q_1 d_{2i+1} \\
&\phantom{=} + Q_1 b_{2i-1} + Sq^1Q_1 v(b_{2i-1}) \\
&\phantom{=} + 0.
\end{align*}
The $Q_1$ component implies that $b_{2i-1} = c_{2i+1}$.  But then,
the $Sq^1Q_1$ component is $Sq^1Q_1(d_{2i+1}+v(c_{2i+1})) \neq 0$,
which is a contradiction.

\end{proof}

\begin{proof}[Proof of~\ref{thm:proofofYustheorem} continued]
Now we can finish the proof.  Certainly $V_{-3}$ and $V_{-2}$ cannot 
both be nonzero, since the first implies $a_{-1}=0$ and the second
implies $a_0 = 0$, but we must have $a_{-1} + a_0 = 1$ by 
Lemma~\ref{xconditions}.

If both are 0, then $ M|_{E(1)} \iso P_0$.
Lemma~\ref{xconditions} gives $a_{2i-1} + a_{2i} = 1$,
while $0 = Sq^1Q_1 x_i = Sq^2Sq^2 x_i$ gives $a_i a_{i+2} = 0$.
The entire $\cA(1)$ action is thus determined by $a_{-1}$.
It follows that $M \iso P_0$ or $M \iso \Sigma^{-2}P_1$.

If $V_{-3} \neq 0$, then $y = d_{-1} \neq 0$, while $c_{-1} = 0$
since $V_{-2}=0$.  Also, $a_{-1} = 0$ and $s(y) = 1$.  Therefore,
Lemma~\ref{sqtwoaction} gives
\begin{align*}
Sq^2 y &= x_{-1} \\
Sq^2 Sq^1 y &= Q_1 y + x_0 \\
Sq^2 Q_1 y &= x_2.
\end{align*}
With the exception of $Sq^2 x_{-1} = Sq^2Sq^2 y = Sq^1Q_1 y$,
the action of $Sq^2$ on the $x_i$ alternates as in the case $V=0$.
It follows that $M \iso \Sigma^{-6}P_3$ under the isomorphism
which takes $y$ to the bottom class, $111$,
and $x_1$ to the indecomposable in degree 1,  $124+142+421$, in the notation
of Section~\ref{seclocating}.  (See Figure~\ref{Pns}).

Finally, if $V_{-2} \neq 0$, then $V_{-2} = \langle d_0\rangle$,
$a_0 = 0$, $a_{-1} = 1$, and the $b_i$, $c_i$ and $d_i$ are all 0
except for $c_{-1} = d_0$.  We get
\begin{align*}
Sq^2 d_0 &= x_{0} \\
Sq^2 Sq^1 d_0 &= Q_1 d_0 \\
Sq^2 Q_1 d_0 &= 0,
\end{align*}
while
\begin{align*}
Sq^2 x_{-1} &= x_{1} + Q_1 d_0 \\
Sq^2 x_0 &= Sq^1Q_1 d_0.
\end{align*}
The remaining $Sq^2 x_i$ are as in $P_0$.  This is isomorphic to 
$\Sigma^{-4} P_2$ by the isomorphism under which $d_0$  generates
the Joker, while $R$ is the submodule spanned by
\[
x_{-1}+Sq^1 d_0,\,\, x_0,\,\, x_1 + Q_1 d_0,\,\,  x_2 + Sq^1 Q_1 d_0,\,\, 
x_3,\,\, x_4, \ldots
\]
\end{proof}

\section{$\Pic^{(k)}$ continued}

We now turn to the determination of the groups $\Pic^{(0)}$.
For $E(1)$, the argument  is similar to that
for $\Pic^{(1)}$.

\begin{prop} 
\label{PiczeroE}
If $M \in E(1)\Mod^{(0)}$ and $H(M,Q_0) = \Sigma^s \F_2$ then
$M \simeq \Sigma^{s+1}R$.  Therefore,
$\Pic^{(0)}(E(1)) = \{\Sigma^i R\} \iso \Z$.
\end{prop}
 
\begin{proof}
By suspending appropriately, we may assume that
$M \in E(1)\Mod^{(0)}$ and $H(M,Q_0) = \Sigma^{-1}\F_2$.
We may also assume that $M$ is reduced.

Let $0 \neq \langle[x]\rangle = H(M,Q_0)$, so that $Sq^1x = 0$ and $x\notin \Image(Sq^1)$.
There are two possibilities:
\begin{enumerate}
\item $Q_1 x = 0$
\item $Q_1 x \neq 0$
\end{enumerate}
In the first case, $x = Q_1 y_0$ for some $y_0$, which cannot be in the image of
$Sq^1$, since $M$ is reduced, so that $Sq^1 y_0 \neq 0$.  We may then assume for induction
that we are 
given $y_i$ such that $Q_1 y_{i} = Sq^1 y_{i-1}$ for $0 \leq i \leq n$, 
and such that $Q_1 y_0 = x$ and $Sq^1 y_n \neq 0$.  The assumption that $M$ is reduced
allows us to extend this another step, completing the induction.  We conclude that
$M$ is not bounded-below, contrary to assumption.

It therefore follows that $Q_1 x \neq 0$.
Then $Sq^1 Q_1 x = 0$ because $M$ is reduced, so $Q_1 x = Sq^1 x_1$
for some $x_1$.  
Again,  $M$  reduced implies that
$Q_1x_1 \neq 0$.  We may assume for induction that we have elements
$x_i$ with $Sq^1 x_i = Q_1 x_{i-1} \neq 0$ for $0 \leq i \leq n$ and $Q_1 x_n \neq 0$.
(Let $x_0 = x$ here.)
Then $M$ reduced implies $Q_1 x_n = Sq^1 x_{n+1}$ for some $x_{n+1}$ and
$Q_1 x_{n+1} \neq 0$, completing the induction.  The $x_i$ generate a
submodule isomorphic to $R$ and the inclusion $R \lra M$ induces a
stable isomorphism.

Now suppose that $M\in \Pic^{(0)}(E(1))$.  By Lemma~\ref{onedim},
$H(M,Q_0) = \Sigma^s\F_2$ for some $s$,
and therefore $M \simeq \Sigma^{s+1} R $.
Finally, observe that the $\Sigma^i R$ are all distinct because 
$H(\Sigma^i R,Q_0) = \Sigma^{i-1} \F_2$.  
\end{proof}

For $\cA(1)$, the argument is a bit more complicated, but the conclusion
is the same.

\begin{prop} 
\label{PiczeroA}
If $M \in \cA(1)\Mod^{(0)}$ and $H(M,Q_0) = \Sigma^s \F_2$ then
$M \simeq \Sigma^{s+1}R$.  Therefore,
$\Pic^{(0)}(\cA(1)) = \{\Sigma^i R\} \iso \Z$.
\end{prop}
 
\begin{proof}
By suspending appropriately, we may assume that
$M \in \cA(1)\Mod^{(0)}$ and $H(M,Q_0) = \Sigma^{-1}\F_2$.
We may also assume that $M$ is reduced:
 $Sq^2 Sq^2 Sq^2$ acts as $0$ on $M$.
Let $0 \neq \langle[x]\rangle = H(M,Q_0)$, so that $Sq^1x = 0$ and $x\notin \Image(Sq^1)$.

Let $M \iso M_0 \dirsum M_1$ as an $E(1)$-module,
where $M_0$ is a reduced $E(1)$-module and $M_1$ is $E(1)$-free.
Then $M_0$ is in $\Pic^{(0)}(E(1))$ with 
$H(M_0,Q_0) = H(M,Q_0) = \langle [x] \rangle$.
By the preceding Proposition, $M_0 \iso R$.  
We may choose $x \in M_0$.
In particular, $Q_1 x \neq 0$.
Since $Sq^1 x = 0$, $Q_1 x \neq 0$ implies that  $Sq^1 Sq^2 x \neq 0$.
There are two possibilities:
\begin{enumerate}
\item $Sq^2Sq^1Sq^2 x \neq 0$
\item $Sq^2Sq^1Sq^2 x = 0$
\end{enumerate}
In the first case, the submodule $\langle x\rangle$ is $\Sigma^{-1}\cA(1)\modmod\cA(0)$ since
$M$ is reduced and $Sq^1x=0$.  The long exact sequences in
$Q_k$-homology induced by the short exact sequence
\[
0 \lra \langle x \rangle \lra M \lra M/\langle x\rangle \lra 0
\]
imply that $H(M/\langle x \rangle, Q_1) = 0$ and 
$H(M/\langle x \rangle, Q_0) = \langle[y]\rangle$ with $Q_0y = Sq^2Sq^1Sq^2 x$.
Then $M/\langle x\rangle$ satisfies the same hypotheses as $M$ shifted up by 4 degrees.
We can thus inductively construct $R \lra M$ inducing an isomorphism in
$Q_0$ and $Q_1$ homology.  Hence $M$ is stably isomorphic to $R$ as claimed.

The second alternative implies that the submodule generated  by $x$ is
spanned by $x$, $Sq^2x$ and $Sq^1Sq^2x$.  This has $Q_1$ homology $\langle[Sq^2x]\rangle$.
The long exact homology sequence for
\[
0 \lra \langle x \rangle \lra M \lra M/\langle x\rangle \lra 0
\]
then implies that $H(M/\langle x\rangle,Q_0) = 0$ and
$H(M/\langle x\rangle,Q_1) = \langle[y]\rangle$ with $Q_1 y = Sq^2x$.  By Yu's theorem
(Theorem~\ref{Yustheorem}), $M/\langle x\rangle$ must be a suspension of $P_n$ for
some $n$.  (It is actually isomorphic to $\Sigma^iP_n$, not just stably equivalent
to it, because it is reduced,
being a quotient of the reduced module $M$.)
Further, if we let $y\in M$ be a class whose image in $M/\langle x \rangle$
generates $H(M/\langle x \rangle, Q_1)$
then $Q_1 y = Sq^2 x$.  Now $Sq^1 y = 0$ because this is so in each $P_n$ and because
$x$, which is in the same degree as $Sq^1y$, is not in the image of $Sq^1$.  Thus,
we must have $Sq^1 Sq^2 y = Sq^2 x$.  This is impossible. In $P_0$, $Sq^2y=0$,
while in $P_n$, $1\leq n \leq 3$, $Sq^2y$ is in the image of $Sq^1$.  Since
$\langle x \rangle$ is zero in this degree, the same holds in $M$.  This contradiction
shows that the second alternative does not happen, proving the theorem.

Now suppose that $M\in \Pic^{(0)}(\cA(1))$.  By Lemma~\ref{onedim},
$H(M,Q_0) = \Sigma^s\F_2$ for some $s$,
and therefore $M \simeq \Sigma^{s+1} R $.
Finally, observe that the $\Sigma^i R$ are all distinct because 
$H(\Sigma^i R,Q_0) = \Sigma^{i-1} \F_2$.  
\end{proof}

From these last four results, we have the converse of Lemma~\ref{onedim}.

\begin{corollary}
\label{onedimconverse}
A module $M \in B\Mod^{(k)}$ is in $\Pic^{(k)}(B)$ iff $H(M,Q_k)$ is one dimensional.
\end{corollary}

It is useful to have explicit forms for these isomorphisms between 
the torsion-free quotient of $\Pic^{(k)}(B)$ and $\Z$.

\begin{corollary}
For $M \in \Pic^{(k)}(B)$, let $d_k(M)$ be defined by $H(M,Q_k) = \Sigma^{d_k(M)}\F_2$.
Then $d_k : \Pic^{(k)}(B) \lra \Z$ is a homomorphism.   It is an isomorphism if
$k=0$ or $B=E(1)$.  When $k=1$ and $B=\cA(1)$, $\Ker(d_1) = \{\Sigma^{-2i}P_i\}
\iso \Z/(4)$.
\end{corollary}

\begin{proof}
The K\"unneth isomorphism implies that $d_k$ is a homomorphism.
The remainder follows directly from Theorems~\ref{Yustheorem} and
\ref{Yuslemma}, and Propositions~\ref{PiczeroE} and \ref{PiczeroA}.
\end{proof}

When $k=1$ and $B=\cA(1)$ we need another invariant to detect $\Ker(d_1)$.
It is possible to define it directly in terms of $M$
by considering divisibility of elements in 
$\Ext_{\cA(1)}^1(M,\F_2)$, but this is cumbersome to define, so
we content ourselves with an invariant defined in terms of $M^{\red}$.

\begin{prop}
If $M \in \Pic^{(1)}(\cA(1))$,  let
\begin{itemize}
\item $c$ be the connectivity (bottom non-zero degree) of $M^\red$,
\item $e = \dim(Sq^2(M^\red_c))$, and
\item $f = \dim(Sq^2Sq^2(M^\red_c))$.
\end{itemize}
(Here $\dim$ refers to dimension as an $\F_2$ vector space.)
Let $t_1(M) =d_1(M)-c-e+f$.
Then $t_1: \Pic^{(1)}(\cA(1)) \lra \Z/(4)$
 is a homomorphism and
$M \simeq \Sigma^{d_1(M)-2t_1(M)} P_{t_1(M)}$.
\end{prop}

\begin{proof}
It is simplest to reverse engineer this.  We compute these invariants
for $\Sigma^i P_n$:

\begin{center}
\begin{tabular}{|l|rrrr|}
\hline
 & $\Sigma^i P_0$
 &$ \Sigma^i P_1$
 &$ \Sigma^i P_2$
 &$ \Sigma^i P_3$
\\
\hline
$c$ & i-1 & i+1 & i+2 & i+3 \\
$d_1$ & i & i+2 & i+4 & i+6 \\
$e$ & 1 & 0 & 1 & 1 \\
$f$ & 0 & 0 & 1 & 1 \\
\hline
$t_1 = d_1-c-e+f$ & 0 & 1 & 2 & 3 \\
\hline
\end{tabular}
\end{center}
Theorem~\ref{multmain} shows that $t_1$ is a homomorphism.  The 
equivalence  between
$M$ and $ \Sigma^{d_1(M)-2t_1(M)} P_{t_1(M)}$ is evident from
the table above.
\end{proof}

\section{The homomorphisms from $\Pic$ to $\Pic^{(k)}$}
\label{sec:Pichom}

Over a finite dimensional graded Hopf algebra, the Picard group always contains
suspension and loops.  This accounts for the $\Z\dirsum\Z$
found by Adams and Priddy (Theorems~\ref{AdamsandPriddyE}
and \ref{AdamsandPriddy}) in $\Pic(E(1))$ and $\Pic(\cA(1))$.
In the Picard groups of the localized subcategories $B\Mod^{(k)}$
these become dependent:  
$\Sigma(\LL_0\F_2) = \Omega(\LL_0\F_2) $  and
$\Sigma^3(\LL_1\F_2) = \Omega(\LL_1\F_2) $  over $E(1)$,
for example.

Together, the functors $\LL_i$  give an embedding of
$\Pic$ into the localized Picard groups.

\begin{prop}
\label{piceone}
Each $\LL_k :\Pic(E(1)) \lra \Pic^{(k)}(E(1))$ is an epimorphism.  Their product
$\LL$, mapping  $\Pic(E(1))$ to $ \Pic^{(0)}(E(1)) \dirsum \Pic^{(1)}(E(1))$,
is a monomorphism with cokernel $\Z/(2)$.
With respect to the basis $\{ \Sigma \F_2, \Omega \F_2\}$ of $\Pic$
we have
\[
\xymatrix{
\Pic(E(1)) 
\ar^{\LL}[d]
\ar^{\,\,\,\,\begin{bmatrix} 1 & 1 \\ 1 & 3 \end{bmatrix}}[rrd]
&
&
\\
\Pic^{(0)}(E(1)) \dirsum \Pic^{(1)}(E(1))
\ar_<<<<<<<<{d_0 \dirsum d_1}[rr]
& & 
\Z \dirsum \Z
\\
}
\]
\end{prop}

\begin{proof}
Explicitly, $\LL(M) = (\LL_0M,\LL_1M) = (\Sigma R \tensor M, P_0 \tensor M)$.
We simply compute:  
\[
d_0(\Sigma R \tensor \Sigma \F_2) =  d_0(\Sigma^2 R) = 1
\]
and
\[
d_1(P_0 \tensor \Sigma \F_2) = d_1(\Sigma P_0) = 1,
\]
while
\[
d_0(\Sigma R \tensor \Omega \F_2) = d_0(\Omega \Sigma R)  = 
d_0(\Sigma^2 R) = 1
\]
and
\[
d_1(P_0 \tensor \Omega \F_2) = d_1(\Omega P_0) = d_1(\Sigma^3 P_0) = 3.
\]
\end{proof}

Over $\cA(1)$ we also have the torsion summands to consider.

\begin{prop}
\label{picaone}
The restriction maps 
\[
\Pic(\cA(1))\lra \Pic(E(1))
\]
 and
\[
\Pic^{(k)}(\cA(1))\lra \Pic^{(k)}(E(1))
\]
 induce isomorphisms from
the torsion free quotients of their domains to their co\-do\-mains, and
commute with $\LL$.
Each $\LL_k : \Pic(\cA(1)) \lra \Pic^{(k)}(\cA(1))$ is an epimorphism.
With respect to the basis $\{ \Sigma \F_2, \Omega \F_2, J\}$ of $\Pic$,
the homomorphism
$\LL : \Pic(\cA(1)) \lra \Pic^{(0)}(\cA(1)) \dirsum \Pic^{(1)}(\cA(1))$
is
\[
\xymatrix{
\Pic(\cA(1)) 
\ar^{\LL}[d]
\ar^{\,\,\,\,\begin{bmatrix} 1 & 1 & 0 \\ 1 & 3 & 0\\ 0 & \overline{1} & \overline{2}
\end{bmatrix}}[rrd]
&
&
\\
\Pic^{(0)}(\cA(1)) \dirsum \Pic^{(1)}(\cA(1))
\ar_>>>>>>>>>>>{\begin{bmatrix} d_0 & 0 \\ 0 & d_1\\ 0 & t_1
\end{bmatrix}}[rr]
& & 
\Z \dirsum \Z \dirsum \Z/(4)
\\
}
\]
with $\overline{k}$ denoting the coset $k + (4)$.  The cokernel of $\LL$ is
$\Z/(4)$.
\end{prop}

\begin{proof}
Again, we simply compute.  The $d_0$ and $d_1$ calculations are the same as for $E(1)$.
This implies the first claim and gives the upper left two by two submatrix.  For the
remainder, we first compute $\LL_1$.   We have
$\LL_1(\Sigma \F_2) = \Sigma P_0$, which projects to $\overline{0}$
in the $\Z/(4)$ summand.  
We also have
$\LL_1(\Omega \F_2) = \Omega P_0 = \Sigma^1 P_1$, which projects to $\overline{1}$
in the $\Z/(4)$ summand.   Next, 
\[
d_0(\LL_0(J)) = d_0(\Sigma R \tensor J) = 0
\]
\[
d_1(\LL_1(J)) = d_1(P_0 \tensor J) =  0.
\]
Finally,  $P_0 \tensor J $ is stably isomorphic to $\Sigma^{-4} P_2$.  
This follows by tensoring the short exact sequence containing $M_2 = \Sigma^4 J$
of Theorem~\ref{main} with $P_0$.  Since $P_0 \tensor R$ is free by
Theorem~\ref{stableiso}, this
gives an equivalence $P_0 \tensor J = P_0 \tensor \Sigma^{-4} M_2
\hequiv P_0 \tensor \Sigma^{-4} P_2 
\hequiv  \Sigma^{-4} P_2 $.

Determination of the cokernel is a simple Smith Normal Form calculation.
\end{proof}

\section{Idempotents and localizations}
\label{sec:idempotents}

Again let $B$ be either $E(1)$ or $\cA(1)$.
In this section we show
that $\LL_0$ and $\LL_1$ are essentially unique, in that the only 
stably idempotent modules in $B\Mod^b$ are ones we have already seen.

\begin{theorem}
\label{idempotents}
If $M \in B\Mod^b$ is stably idempotent then $M$ is stably equivalent to
one of $0$,
$\F_2$, $P_0$, $\Sigma R$,
or $P_0 \dirsum \Sigma R$.
\end{theorem}

\begin{proof}
We give the proof for $B=\cA(1)$.  The proof for $E(1)$ is similar but easier.

We first note a simple fact:  if $M \tensor M \simeq M$ then
each $H(M,Q_i)$ must be either $0$ or $\F_2$.  This yields four possibilities.

If both are $0$, then $0 \lra M$ is a stable equivalence by Theorem~\ref{stableiso}.

If exactly one $Q_i$-homology group is nonzero, we have the unit in
$\Pic^{(i)}(\cA(1))$, which 
must be either $P_0$ or $\Sigma R$ by
Theorems~\ref{Yustheorem} and \ref{PiczeroA}.

The final possibility is that $H(M,Q_0) = \F_2 = H(M,Q_1)$.  In this case
we tensor $M$ with the triangle
\begin{equation*}
\label{nonsplit}
\Sigma R \lra \F_2 
\lra P_0
\end{equation*}
We get
a triangle
\[
 \LL_0M \lra M
\lra \LL_1M.
\]
By Theorem~\ref{Lsubiproperties}, each
$\LL_i(M)$ is stably idempotent and $Q_i$-local.
By the preceding paragraph, $\LL_0(M) \simeq \Sigma R$ and $\LL_1(M) \simeq P_0$.
It remains to determine the possible extensions ${M}$.

It is a simple matter to verify that
$ \Ext_{\cA(1)}^{1,0}( P_0, \Sigma R) = \F_2$.
Therefore, the two possibilities are
the split extension $M \simeq P_0 \dirsum \Sigma R$ and
the nonsplit $M \simeq \F_2$  above.
\end{proof}

\section{A final example}
\label{final}

As noted in
\ref{stableisofail}, the detection of stable isomorphisms is more subtle
in the category of all $\cA(1)$-modules:  the module $L = \F_2[x,x^{-1}]$
has trivial $Q_0$ and $Q_1$ homology, yet is not stably free. It provides
another idempotent as well.

\begin{prop}
As $\cA(1)$-modules, 
$\displaystyle{L \tensor L \iso  L 
\dirsum \bigdirsum_{i,j \in \Z} \Sigma^{4i+2j-2} \cA(1)}$. 
\end{prop}

\begin{proof}
The elements $x^{4i-1}\tensor x^{2j-1}$
generate a free submodule, and the
submodule $\{ x^i \tensor x^0 | i \in \Z\}$, which is isomorphic to $L$,
is a complementary submodule.
\end{proof}

Therefore, we have another localization functor 
\[
\LL_\infty(M) =  L \tensor M.
\]

The module $L$ shows that
 $Q_0$ and $Q_1$ homology are insufficient to capture
a more general notion of being $Q_0$ or $Q_1$ local.

\begin{prop}
$\LL_0 L \hequiv 0$ and $\eta_L : L \llra{\hequiv} \LL_1 L$.
\end{prop}

\begin{proof}
$\LL_0 L = \Sigma R \tensor L$ is free over $\cA(1)$ on
the elements $x^{4i-1}\tensor x^{2j-1}$ with $i \geq 0$.
The canonical triangle, $\LL_0 L \lra L \llra{\eta_L} \LL_1 L$
then shows that $L$ is equivalent to its $\LL_1$ localization.
\end{proof}




\appendix

\section{Locating $P_n$ in $P^{\tensor(n)}$}
\label{seclocating}

The `hit problem' 
is the problem of determining a set
of $\cA$-module generators of the polynomial rings
$\F_2[x_1,\ldots,x_n] = H^*B(C_2^n)_{+}$.  
See \cite{AultSinger} for a recent paper on the problem, and
\cite{Ault} for work on the problem using the results we
prove here.
One approach to it is
to consider the analogous problem over subalgebras $\cA(n)$.
The results of section~\ref{PntoPone} simplify the problem
in the case of $\cA(1)$.  Those results only identify the stable type, $P_n$,
of $H^*(BC_2 \smsh \cdots \smsh BC_2)$.  
In this section we will produce
explicit embeddings 
$P_n \lra P^{\tensor(n)}$.  
Naturally, there are choices involved,
but the inductive determination of the isomorphism type also gives us
a way to inductively
find $P_{n+1}$ as a summand of $P_n \tensor P \subset P^{\tensor(n)} \tensor P$,
reducing the work dramatically.

Let us write $x_1^{i_1} \ldots x_n^{i_n}$ as $i_1 \ldots i_n$ and define
$\overline{i_1 \ldots i_n}$ to be the orbit sum of
${i_1 \ldots i_n}$.

\begin{theorem}
\label{location}
For $n>0$,
$P_n$ can be embedded in $P^{\tensor(n)}$ as follows:
\begin{itemize}
\item $M_1 = \langle 1, 2, 4 \rangle$
\item $P_1 = P = M_1 + 
 \langle 3, i \st i \geq 5 \rangle$
\item $M_2 = \langle 11, \overline{12}, 22, \overline{14}, \overline{24} \rangle$
\item $P_2 = M_2 + \langle 21, 4i \st i \geq 1 \rangle$
\item $M_3 = \langle  \overline{112},
222+\overline{114}, 
\overline{124} \rangle$
\item $P_3= M_3 + \langle
111, \overline{122}, 222, \overline{224}
124+142+421,
44i \st i \geq 1 \rangle$

\item $M_4 = \langle 2222 + \overline{1124} \rangle$
\item $P_4 = M_4 + \langle
2221 + \overline{1114},
\overline{122}4 + \overline{124}2,
2224,
\overline{224}i \st i \geq 1 \rangle$
\item $P_{n+4} \iso 
\langle 2222 + \overline{1124} \rangle 
 \tensor 
 P_n$.
\end{itemize}
\end{theorem}

\begin{remark}
\label{choiceremark}
There are several notable points about these submodules.
\begin{enumerate}
\item The first three generators
$x_1$, $x_1 x_2$, and $x_1 x_2 x_3$,
are obvious from the connectivity:  the connectivity of $P_n$ is $n$ for $n<4$.
\item \label{choices}
The fourth,
$x_1^2 x_2^2 x_3^2 x_4 + \overline{x_1 x_2 x_3 x_4^4}$ in degree 7, is less so.
The classes of degrees less than 7 all lie in free summands of $P_3 \tensor P$.
Modulo those free summands, 
there are 4 possible choices for the degree 7 class in $P_4$:
\[
2221 + \overline{1114} + \alpha_0(2221+\overline{221}2)
+\alpha_1(1114 + \overline{112}3)
\]
for $\alpha_i \in \{0,1\}$.  
\item Applying $Sq^1$ to any of these four classes yields the same
`periodicity class' $B=2222 + \overline{1124}$ in degree 8.  From $H(P,Q_1) = [x_1^2]$,
we know that $H(P_4,Q_1) = [x_1^2 x_2^2 x_3^2 x_4^2]$, but since 
$Sq^2(x_1^2 x_2^2 x_3^2 x_4^2) \neq 0$, the `periodicity class' must have additional
terms, which turn out to be exactly $Q_0Q_1(x_1x_2x_3x_4)$, or 
$\overline{1124}$ in our abbreviated notation.
\item Above the bottom few degrees, each of the $P_i$ can be written as the tensor
product of an $\cA(1)$-annihilated class with $P$.  These $\cA(1)$-annihilated
classes are $B^i$, $x_1^4B^i$, $x_1^4x_2^4B^i$, and $\overline{x_1^2 x_2^2 x_3^4}B^i$.
\end{enumerate}
\end{remark}

\begin{proof}
Evidently $P_1=P$.
For $P_2$, it is a simple matter to verify that $x_1x_2$ generates $M_2$.
To finish $P_2$, 
clearly $x = x_1^2 x_2$ serves, with the rest of $P_2$  then given by $x_1^4(x_2^i)$.

Expressing $P_3$ as the nontrivial extension of 
$M_3$ by $P_3/M_3 \iso \Sigma^4 R$ requires
that the bottom class of $M_3$ be $Sq^1(x_1x_2x_3) = \overline{112}$.
The bottom $\cA(1)\modmod\cA(0)$ is forced, but for the second one,
we need $x$ with $Sq^1x = \overline{x_1^2 x_2^2 x_3^4}$.  By choosing
$x = x_1 x_2^2 x_3^4 + x_1 x_2^4 x_3^2 + x_1^4 x_2^2 x_3$, the rest of
$P_3$ is given by $x_1^4 x_2^4(x_3^i)$.

For $P_4$, we need a  class in degree 7 in  $P_3 \tensor P$
which is not in $\Image(Sq^1)+\Image(Sq^2)$ and whose annihilator
ideal is $(Sq^2 Sq^1)$.  Solving $Sq^1 x \neq 0$, $Sq^2 Sq^1 x = 0$,
$Sq^2Sq^1Sq^2 x \neq 0$, for $x \notin
\Image(Sq^1)+\Image(Sq^2)$,
we arrive at the 4 choices in Remark~\ref{choiceremark}.(\ref{choices})
above.  Our choice, $\alpha_0 = \alpha_1 = 0$, gives the version of
$P_4/M_4$ which is simplest to describe.

Finally, consider periodicity.  Since $M_4$ is a trivial $\cA(1)$ module,
tensoring with it is the same as 8-fold suspension.
Now, if we tensor the short exact sequence 
$ 0 \lra M_4 \lra P_4 \lra \Sigma^8 R \lra 0$
with $P_n$, we get
\[
0 \lra M_4 \tensor P_n \lra P_4 \tensor P_n \lra \Sigma^8 R \tensor P_n \lra 0.
\]
The K\"unneth theorem and Theorem~\ref{stableiso} imply that $M_4\tensor P_n$
is stably isomorphic to $P_4 \tensor P_n$, and hence to 
$P^{\tensor(4)}\tensor P^{\tensor(n)}$.
Since $M_4 \tensor P_n$ is indecomposable,
it follows that it
is isomorphic to $P_{n+4}$ and that the inclusion 
$M_4 \tensor P_n \subset P_4 \tensor P_n \subset P^{\tensor(4)} \tensor P^{\tensor(n)}$
serves our purpose.
\end{proof}

\noindent
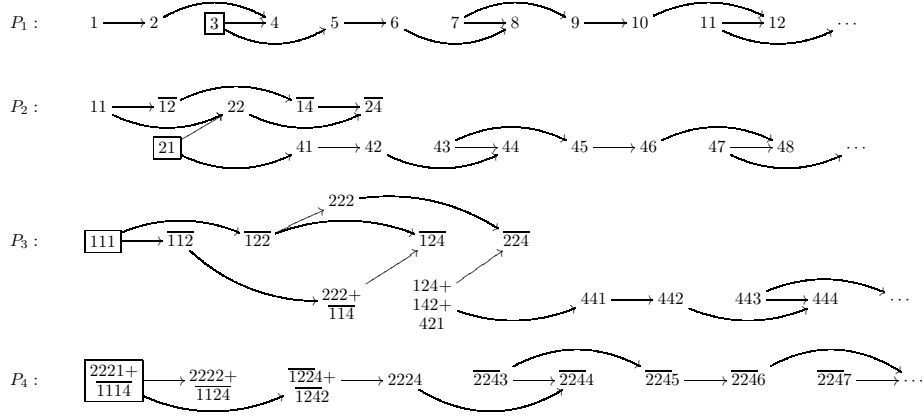
\begin{figure}
\scalebox{.65}{
$
\xymatrix{
P_1: &
1
\ar[r]
& 2
\ar@/^1pc/[rr]
&*+[F]{3}
\ar[r]
\ar@/_1pc/[rr]
& 4
& 5
\ar[r]
& 6
\ar@/_1pc/[rr]
& 7
\ar[r]
\ar@/^1pc/[rr]
& 8
& 9
\ar[r]
&
10
\ar@/^1pc/[rr]
&
11
\ar[r]
\ar@/_1pc/[rr]
&
12
&
\cdots
&&&
\\
\\
}
$
 } 
\\
\scalebox{.65}{
$
\xymatrix@R=9pt{
P_2: &
11
\ar[r]
\ar@/_1pc/[rr]
& \overline{12}
\ar@/^1pc/[rr]
& {22}
\ar@/_1pc/[rr]
& \overline{14}
\ar[r]
& \overline{24}
&
&
&
&
&
&
&
\\
&
&
*+[F]{21}
\ar[ru]
\ar@/_1pc/[rr]
&
& 41
\ar[r]
& 42
\ar@/_1pc/[rr]
& 43
\ar[r]
\ar@/^1pc/[rr]
& 44
& 45
\ar[r]
& 46
\ar@/^1pc/[rr]
& 47
\ar[r]
\ar@/_1pc/[rr]
& 48
&
\cdots
&&
\\
\\
}
$
} 
\\
\scalebox{.65}{
$
\xymatrix@R=9pt{
&&&
& 222
\ar@/^1pc/[rrd]
\\
P_3: &
*+[F]{111}
\ar[r]
\ar@/^1pc/[rr]
&
\overline{112}
\ar@/_1pc/[rrd]
&
\overline{122}
\ar[ru]
\ar@/^1pc/[rr]
&&
\overline{124}
&
\overline{224}
&
\\
&&&& 
\txt{$222+$\\$\overline{114}$}
\ar[ru]
&
{\txt{$124+$ \\ $142+$\\$421$}}
\ar[ru]
\ar@/_1pc/[rr]
&
& 441
\ar[r]
& 442
\ar@/_1pc/[rr]
& 443
\ar[r]
\ar@/^1pc/[rr]
& 444
&
\cdots
&
\\
\\
}
$
} 
\\
\scalebox{.65}{
$
\xymatrix{
P_4: &
*+[F]{\txt{$2221+$\\$\overline{1114}$}}
\ar[r]
\ar@/_1.5pc/[rr]
& \txt{$2222+$\\$\overline{1124}$}
& \txt{$\overline{122}4 +$\\$ \overline{124}2$}
\ar[r]
& 2224
\ar@/_1.5pc/[rr]
& \overline{224}3
\ar[r]
\ar@/^1.5pc/[rr]
& \overline{224}4
& \overline{224}5
\ar[r]
& \overline{224}6
\ar@/^1.5pc/[rr]
& \overline{224}7
\ar[r]
&
\cdots
&
\\
\\
}
$
} 
{\caption {The modules $P_n$
embedded in $P^{\tensor(n)}$, $1 \leq n \leq 4$.
The bottom class of
the quotient $\Sigma^t R$ is boxed.  See Section~\ref{seclocating} for notation.}
\label{Pns}}
\end{figure}

\section{The free summand in $P^{\tensor(n)}$}

We have now shown that  if $n > 0$ then
\[
P^{\tensor n} = P_n \dirsum F_n
\]
where $F_n$ is a free $\cA(1)$-module.
We can therefore give a complete decomposition of 
$P^{\tensor(n)}$
by simply computing
the Hilbert series of the free part.  This can be found in  Yu's thesis 
(\cite[Theorem 4.2]{Yu}).
The most transparent form of the Hilbert series for the $P_n$ can
simply be read off from
Theorem~\ref{main}.

\begin{lemma} $H(P_{4k+i}) = t^{8k} H(P_i)$ and
\begin{itemize}
\item  $H(P_0) =\displaystyle{ \frac{t^{-1}}{1-t}}$\\[1ex]
\item  $H(P_1) = \displaystyle{\frac{t^{1}}{1-t}}$\\[1ex]
\item  $H(P_2) = \displaystyle{\frac{t^{2}}{1-t} + t^3+t^5+t^6}$\\[1ex]
\item  $H(P_3) = \displaystyle{\frac{t^{3}}{1-t} + t^6+t^7}$
\end{itemize}
\end{lemma}

Another form works a bit better in connection with the Hilbert series for $P^{\tensor(n)}$.

\begin{lemma}
The Hilbert series 
\[
H(P_n) = \frac{t^{2n}}{1-t}Q_n(t)
\]
where
\[
Q_n(t) =\left\{
 \begin{array}{ll}
 {\displaystyle{\frac{1}{t}}} & n \congruent 0,1 \pmod{4} \\[2ex]
 {\displaystyle{\frac{1+t-t^2+t^3-t^5}{t^2}}} & n \congruent 2 \pmod{4} \\[2ex]
 {\displaystyle{\frac{1+t^3-t^5}{t^3}}} & n \congruent 3 \pmod{4} \\
 \end{array}\right.
\]
\end{lemma}

\begin{proof}
Straightforward.
\end{proof}

We can now locate the summands in  the free parts  $F_n$.  

\begin{theorem}
\label{freepart}
The Hilbert series of the modules $F_n$ are
\[
H(F_n) = H(\cA(1)) \frac{t^n(1-t^n(1-t)^{n-1}Q_n(t))}
			{(1-t)^{n-1}(1-t^4)(1+t^3)}
\]
\end{theorem}

\begin{proof}
We simply compute 
\begin{eqnarray*}
\frac{H(P^{\tensor n}) - H(P_n)}{H(\cA(1))}
 & = & \frac{\left({\displaystyle{\frac{t}{1-t}}}\right)^n - 
	    {\displaystyle{\frac{t^{2n}}{1-t}}}Q_n}
      {(1+t)(1+t^2)(1+t^3)} \\[2ex]
& = & \frac{t^n - t^{2n}(1-t)^{n-1}Q_n(t)}
           {(1-t)^n(1+t)(1+t^2)(1+t^3)} \\[2ex]
& = & \frac{t^n(1 - t^{n}(1-t)^{n-1}Q_n(t))}
           {(1-t)^{n-1}(1-t^4)(1+t^3)} \\
\end{eqnarray*}
\end{proof}


The following special cases are of particular interest, and are the correct replacement
for  Lemma 2 in \cite{Ossa}, where the free part of $P \tensor P$ is asserted to be
$ \cA(1) \tensor \Sigma^2   \F_2[u_2,v_4].$

\begin{corollary} As   $\cA(1)$-modules
\label{PP}
\[
P \tensor P_0 \iso P \dirsum \left( \cA(1) \tensor \F_2[u_2,v_4]\right)
\]
and
\[
P \tensor P \iso  P_2
    \dirsum \bigdirsum_{\substack{i,j \geq 0\\i+j > 0}} \Sigma^{4i+4j} \cA(1)
    \dirsum \bigdirsum_{i,j \geq 0} \Sigma^{4i+4j+6} \cA(1)
\]
\end{corollary}

\begin{remark}
$P_0$ is the cohomology of $T(-\lambda)$, the Thom complex of the negative of the
line bundle over $P= BC_2$.  As a consequence, the first isomorphism in 
Corollary~\ref{PP} can be used to give a homotopy equivalence
\[
ko \smsh  BC_2 \smsh T(-\lambda) \hequiv (ko \smsh BC_2) \vee H\F_2[u_2,v_4]
\]
\end{remark}

\bibliographystyle{amsplain}

\end{document}